\documentclass[11pt]{amsart}

\usepackage[T1]{fontenc}
\usepackage[utf8x]{inputenc}
\usepackage[pagewise]{lineno} 

\usepackage{amsmath, amssymb}
\usepackage{enumerate}
\usepackage[T1]{fontenc}

\usepackage{theoremref}
\usepackage{mathrsfs}
\usepackage{hyperref}

\usepackage{tikz}
\usepackage{times}
\usepackage{color}
\usepackage{amsmath}
\usepackage{amssymb}
\usepackage{amsthm}
\usepackage{enumerate}
\usepackage{amsbsy}
\usepackage{amsfonts}
\newtheorem{theo}{Theorem}[section]
\newtheorem{defin}[theo]{Definition}
\newtheorem{prop}[theo]{Proposition}

\newtheorem{lemm}[theo]{Lemma}
\newtheorem{rem}[theo]{Remark}

\newtheorem{assumption}[theo]{Assumption}

\newcommand{\al}{\alpha}

\newcommand{\Ga}{\Gamma}

\newcommand{\om}{\omega}

\newcommand{\ep}{\epsilon }

\newcommand{\De}{\Delta}
\newcommand{\de}{\delta}

\newcommand{\pa}{\partial}

\newcommand{\R}{{\mathbb R}^n}
\newcommand{\hR}{{\mathbb R}^n_+}

\newcommand{\Rn}{{\mathbb R}^{n-1}}

\newcommand{\na}{\nabla}

\newcommand{\calC}{{ \mathcal C }}

\newcommand{\bke}[1]{\left( #1 \right)}

\newcommand{\bket}[1]{\left\{ #1 \right\}}
\newcommand{\norm}[1]{\left\Vert #1 \right\Vert}
\newcommand{\abs}[1]{\left| #1 \right|}

\newcommand{\calD}{{ \mathcal D  }}
\newcommand{\calS}{{ \mathcal S  }}
\newcommand{\calT}{{ \mathcal T  }}

\newenvironment{pflem22}{{\par\noindent
            \textbf{Proof of Lemma \ref{lemma240109}}\quad}}{\qed}
            
\newenvironment{pflem24}{{\par\noindent
            \textbf{Proof of Lemma \ref{lemma0709-1}}\quad}}{\qed}

\begin{document}
\baselineskip=15pt

\title[ ]{Singular velocity of the Stokes and Navier-Stokes equations near boundary in the half space}

\maketitle

\begin{center}
  \normalsize
  Tongkeun Chang \footnote{Department of Mathematics,
   Yonsei University Seoul, 03722,
    South Korea, \,\, chang7357@yonsei.ac.kr } and Kyungkeun Kang \footnote{Department of Mathematics,
   Yonsei University Seoul, 03722,
    South Korea, \,\, kkang@yonsei.ac.kr}

\end{center}

\begin{abstract}

Local behaviors near boundary are analyzed for solutions of the Stokes and Navier-Stoke equations in the half space with localized non-smooth boundary data.
We construct solutions of Stokes equations whose velocity field is not bounded near boundary away from the support of boundary data, although velocity and gradient velocity of solutions are locally square integrable. This is an improvement compared to known results in the sense that velocity field is unbounded itself, since previously constructed solutions were bounded near boundary, although their normal derivatives are singular.
We also establish singular solutions and their derivatives that  do not belong to $L^q_{\rm{loc}}$ near boundary with $q> 1$.
For  such examples, there corresponding pressures turn out not to be locally integrable.
Similar construction via a perturbation argument is available to the Navier-Stokes equations near boundary as well.


\end{abstract}


\section{Introduction}

In this paper, we study local regularity of the Stokes and the Navier-Stokes equations  near flat boundary.
Let $B^+_{r}:=\{ x=(x', x_n )\in \Rn \times {\mathbb R}: |x|<r, x_n >0\}$, $ n \geq 2$. We consider the Stokes system near local boundaries.
\begin{equation}\label{Stokes-10}
u_t - \De u+  \na \pi =0\qquad \mbox{div } u =0 \quad \mbox{ in }
\,Q_2^+:=B^+_{2}\times (0, 4),
\end{equation}
where  $ u$  indicates the velocity of the fluid and $\pi$ is the associated pressure. We emphasize that no-slip boundary condition is given only on the flat boundary, i.e.
\begin{equation}\label{bddata-20}
u=0 \quad \mbox{ on } \,\Sigma:=(B_{2}\cap\{x_n=0\})\times (0,4).
\end{equation}
Similar situation can be examined  for the Navier-Stokes equations, namely
\begin{equation}\label{nse-30}
u_t - \De u +  (u\cdot\nabla) u+\nabla  \pi =0,\qquad \mbox{div } u=0 \quad \mbox{ in }
\,Q_2^+
\end{equation}
together with the boundary condition \eqref{bddata-20}.

Unlike the heat equation, due to non-local effect of incompressible fluid flow, we cannot expect local smoothing property for \eqref{Stokes-10} or \eqref{nse-30} near boundary.
Indeed, the first author constructed a weak solution of \eqref{Stokes-10} whose normal derivatives near boundary are unbounded. To be more precise, the singular solution, denoted by $ u$, constructed in \cite{Kang05} is locally bounded weak solutions in $Q_1^+$, namely,
\begin{equation}\label{bound-weak}
\norm{u}_{L^{\infty}(Q_1^+)}+\norm{u}_{L^2(Q_1^+)}<\infty,\qquad \norm{u}_{L^{\infty}(Q_1^+)}=\infty.
\end{equation}
In \cite{KLLT22}, such construction was extended to the Navier-Stokes equations as well and various types of blow-up rates were established near boundary.
We remark that constructed singular solutions in \cite{Kang05} and \cite{KLLT22} are due to nonlocal effect of solutions in the half-space  caused by compactly supportd singular boundary data.

Meanwhile,  in \cite{Seregin-Sverak10}, Seregin and \u{S}ver\'ak constructed a H\"{o}lder continuous solution up to boundary in the form of shear flow whose
 normal derivatives are, however, unbounded near boundary in the half-space. It is remarkable that such a solution
satisfies homogeneous initial and boundary conditions and the solution, however, is not of finite energy in the half space in contrast to ones constructed in \cite{Kang05} and \cite{KLLT22}.

 On the other hand, a certain type of local maximal regularity was proved near boundary via a priori estimate.
 More precisely, in \cite[Proposition 2]{Seregin00} it was shown in three dimensions that
for given $p, q\in (1,2]$ the solution of the  Stokes equations \eqref{Stokes-10}-\eqref{bddata-20} satisfies the following a priori estimate:
For any  $r$ with $p\le r<\infty$
\begin{equation}\label{Seregin-50}
\norm{u_t}_{L^{q}_t L^r_{x}(Q_{\frac{1}{4}}^+)}+\norm{\nabla^2 u}_{L^{q}_t L^r_{x}(Q_{\frac{1}{4}}^+)}+\norm{\nabla \pi}_{L^{q}_t L^r_{x}(Q_{\frac{1}{4}}^+)}
\le C\bke{\norm{ u}_{L^{q}_t W^{1,p}_{x}(Q_{1}^+)}+\norm{\pi}_{L^{q}_t L^p_{x}(Q_{1}^+)}}.
\end{equation}
Furthermore, due to parabolic embedding, it follows that
\begin{equation}\label{Seregin-100}
\norm{u}_{\calC^{\frac{\alpha}{2}}_t \calC^{\alpha}_{x}(Q_{\frac{1}{4}}^+)}
< \infty,\qquad  0< \alpha <2-\frac{2}{q}.
\end{equation}
It is worth to mention that the estimates \eqref{Seregin-50} is optimal, since the authors constructed in \cite{CK23}
an example that satisfies \eqref{Seregin-50}  but any higher integrability of the example fails (see also \cite{KLLT22}). 

We remark very recently that local boundary regularity of solutions was improved for the Stokes system, when boundary condition is given as Navier-boundary condition, not no-slip boundary condition (see \cite{CLT}).

 So far, singular solutions that has been constrructed in above references are bounded, in fact, H\"{o}lder continuous  but the normal derivatives of velocity are singular near boundary in the half space.
One may wonder whether or not there exists a singular solution that is unbounded itself near boundary. 
Our main objective is to construct a weak solution of the Stokes and Navier-Stokes equations sucht that velocity field is singular near boundary in the half space  (the definition of weak  solutions of Stokes equations is given in section \ref{prelim}).
Furthermore, we are also interested in establishing an example so that velocity fileld is bounded but it does not belong to $\calC^{\alpha}$, class  of H\"{o}lder  continuous functions, with any exponent $\alpha>0$ near boundary.
We remark that since we are looking for solutions that are not H\"{o}lder  continuous functions itself, due to above a priori estimate \eqref{Seregin-50}, the associated pressure is unlikely to belong to $L^p_{\rm loc}$, $p>1$, which turns out to be the case in examples of our construction.
Our main resutls will be precisely stated in the next section.

\section{Main results}
\label{SS-half}
\setcounter{equation}{0}



For convenience, we introduce a tensor $L_{ij}$ defined by
\begin{align}\label{L-tensor}
L_{ij} (x,t) & =  D_{x_j}\int_0^{x_n} \int_{\Rn}   D_{z_n}
\Ga(z,t) D_{x_i}   N(x-z)  dz,\quad
i,j=1,2,\cdots, n,
\end{align}
where $\Ga$ and $N$ are  the $n $ dimensional Gaussian kernel and  Newtonian kernel defined by
\begin{align}\label{H-L-10}
\Ga(x,t) = \left\{ \begin{array}{cc}\vspace{2mm}
\frac1{(4\pi t)^{\frac{n}2}} e^{-\frac{ |x|^2}{4t}}& \quad t > 0,\\
0 &\quad t < 0,
\end{array}
\right.
\quad N(x) = \left\{\begin{array}{ll}\vspace{2mm}
-\frac1{n(n -2)\om_n} |x|^{-n +2} &\quad n \geq 3,\\
\frac1{2\pi} \ln |x| &  \quad n =2.
\end{array}
\right.
\end{align}
It is known that $L_{ij}$ satisfies following relations:
\begin{align} \label{1006-3}
\sum_{i=1}^{n} L_{ii} = \frac12 D_{x_n} \Ga, \quad L_{ij} = L_{ji} \,\, 1 \leq i,j \leq n-1, \quad   L_{in} =
L_{ni}  + B_{in} \,\, \mbox{ if }\, i \neq n,
\end{align}
where
\begin{equation}\label{B-tensor}
B_{in}(x,t) = \int_{\Rn}D_{x_n} \Ga(x^\prime -z^\prime ,
x_n, t) D_{z_i} N( z^{\prime},0) dz^\prime.
\end{equation}
Furthermore, we remind an estimate of $L_{ij}$ defined in
\eqref{L-tensor} (see (2.6) in \cite{Sol02})
\begin{equation}\label{est-L-tensor}
|D^{l_0}_{x_n} D^{k_0}_{x'}  L_{ij}(x,t)| \leq
\frac{c}{t^{  \frac12} (|x|^2 +t )^{\frac12 n + \frac12 k_0}
(x_n^2 +t)^{\frac12 l_0}},
\end{equation}
where $ 1 \leq  i , j\leq n$ except for the case that $i \neq n$ and $j=n$.

We recall the following Stokes equations in a  half space:
\begin{align}\label{StokesRn+}
\begin{split}
& w_t - \De w +  \na p =0,\qquad \mbox{div }\, w =0 \quad \mbox{ in }
\, \R_+ \times (0,\infty),\\
& \qquad\quad \displaystyle w|_{t =0} =0, \qquad w|_{x_n =0} =g.
\end{split}
\end{align}
It is known that $w$ is represented by
\begin{align}\label{rep-bvp-stokes-w}
w_i(x,t) & = \sum_{j=1}^{n}\int_0^t \int_{\Rn} K_{ij}(
x^{\prime}-y^{\prime},x_n,t-s)g_j(y^{\prime},s) dy^{\prime}ds,
\end{align}
where  the Poisson kernel $K $ of the Stokes equations in $\R_+\times (0, \infty)$ is  given as follows   (see \cite{Sol02}):
\begin{align}\label{Poisson-tensor-K}
\begin{split}
  K_{ij}(x'-y',x_n,t) &  =  -2 \delta_{ij} D_{x_n}\Ga(x'-y', x_n,t)
+4L_{ij} (x'-y',x_n,t)\\
& \qquad +2  \de_{jn} \de(t)  D_{x_i} N(x'-y',x_n).
\end{split}
\end{align}
The formula of the associated pressure $p$ is also given as 
\begin{align}\label{representationpressure}
\begin{split}
p(x,t) = & - \sum_{k =1}^{k =n}\int_{\Rn} D_{x_k} D_{x_n} N(x' -y', x_n)    g_k(y', t) dy'  \\
& +    2\sum_{k =1}^{k=n}
 (D_t-\De') \int_0^t  \int_{\Rn} g_k (z',s) \int_{\Rn} D_{y_k} N(y', x_n) \Ga(x'-y' -z',0,t-s) dy' dz'ds\\
&  + N*' D_t g_n (x,t),
\end{split}
\end{align}

We specify a boundary data $g:\mathbb R^{n-1}\times \mathbb
R_+\rightarrow  \R$ with only non-zero $n-$th component defined
as follows:
\begin{align}\label{0502-6}
g(y', s) = \al (0, \cdots, 0,  g_n (y',s))=\al (0, \cdots, 0,g^{\calS}_n (y')g^{\calT}_n(s)),
\end{align}
where $\al > 0$ is a parameter, which will be taken sufficiently small for the case of 
the Navier-Stokes equations. Here $g_n^{\mathcal S} \geq 0$ and $g_n^{\mathcal T} \geq 0$ are assumed to satisfy
\begin{align}\label{boundarydata}
g^{\mathcal S}_n \in C^\infty_c (A), \qquad
 {\rm supp} \,\, g_n^{\mathcal T} \subset (0,1), \quad   g^{\mathcal T}_n   \in L^1({\mathbb R}) \cap C^\infty([0, 1)),
\end{align}
where the set $A\subset\Rn$ is defined by
\begin{align}\label{0330-1}
A = \{ y'=(y_1, y_2, \cdots, y_{n-1})\in \Rn \, | \, 1 < |y'| < 2, \,  -2 < y_i
< -1, \, \, 1 \leq i \leq n-1 \}.
\end{align}

Here we consider the case of non-smooth $g^{\mathcal T}_n$, which is supposed to be three kinds of prototype types. More precisely, we assume the following for $g^{\mathcal T}_n$:
\begin{assumption}\label{Assume-phi}
We suppose that  $  g^{\mathcal T}_n  (s) = \phi(1-s)  $, where  $\phi$ satisfies
\begin{align}
\phi \geq 0, \quad {\rm supp} \,\, \phi \subset (0,1), \quad   \phi   \in L^1({\mathbb R}) \cap C^\infty((0, 1]), \quad \phi(1) =0.
\end{align}
Assume further that either
\begin{eqnarray}\label{phi-assume-log}
\phi(s)=\abs{\ln s}\quad \mbox{or}\quad \phi(s)=\abs{\ln s}^{-1},\qquad s\in (0, \frac{1}{2}]
\end{eqnarray}
or $\phi$ satisfies
\begin{eqnarray}\label{phi-assume-v10}
 \phi^{(k)}( s)\approx  \phi^{(k)}( 2s), \quad \qquad s\in (0, \frac{1}{4}],\,\,\,\,k = 0, \,\, 1.
\end{eqnarray}
Here, $A\approx B$ means that   there exists a  constant $c>1$ such that $c^{-1}B\le A\le cB$.
Furthermore, we also assume that 
 \begin{eqnarray}\label{phi-assume-v15}
 &\displaystyle\int_\ep^1 s^{-\al} \phi(  s) ds \approx  \max( 1,  \ep^{ 1-\al} \phi (\ep)), 
 \quad  0  <  \al, \quad \al \neq 1.\\
 \label{phi-assume-v16}
 &\displaystyle\int_\ep^1 s^{-\al} \phi' (  s) ds \approx  \ep^{ 1-\al}  \phi' (\ep),    \quad  0  <  \al \quad \al \neq 1.
 \end{eqnarray}
 \begin{eqnarray}
 \label{phi-assume-v17}
 \displaystyle \int_0^\ep   \phi(s) ds \approx   \ep \phi(\ep).
 \end{eqnarray}
\end{assumption}
Some comments are prepared for the Assumption \ref{Assume-phi}.
\begin{itemize}
\item[(i)]
 In case that $\frac34 <  t < 1$, if $\phi$ satisfies \eqref{phi-assume-v10}, then one can see that
 \begin{eqnarray}\label{phi-assume-v11}
 \phi^{(k)}(t-s)\approx \phi^{(k)}(1-s)\quad \mbox{ for } \quad   \frac12< s < 2t-1\quad k = 0,1,
 \end{eqnarray}
 \begin{eqnarray}\label{phi-assume-v12}
 \phi^{(k)}(1-t)\approx \phi^{(k)}(1-s)\quad \mbox{ for } \quad  2t-1< s <  t \quad k =0,1,
 \end{eqnarray}
since $\frac12 ( 1-s) \leq  t-s \leq  (1-s) $ for $\frac12< s < 2t-1$  and
 $\frac12( 1-s) \leq 1-t \leq  (1-s)$ for $2t-1< s <  t$. Similarly, if $ 1 < t< \frac54$,
it is straightforward that
\begin{eqnarray}\label{phi-assume-v13}
 \phi^{(k)}(t-s)\approx \phi^{(k)}(1-s)\quad \mbox{ for } \quad  \frac34< s < 2-t \quad k =0,1,
 \end{eqnarray}
 \begin{eqnarray}\label{phi-assume-v14}
 \phi^{(k)}( t-1)\approx \phi^{(k)} (t-s)\quad \mbox{ for } \quad  2-t< s <  1 \quad k =0,1.
 \end{eqnarray}
This is also due to facts that
$( 1-s) \leq  t-s \leq  2(1-s)$ and $\frac12( t-s) \leq t-1 \leq  (t-s)$ for $0< s < 2-t$ and $ 2-t< s <  1$, respectively.

\item[(ii)] One can check that $\phi(s)=\abs{\ln s}$ or $\phi(s)=\abs{\ln s}^{-1}$ satisfy
\eqref{phi-assume-v15}-\eqref{phi-assume-v17}. We note that $\phi(t)=t^{-a}$ with $a\in (0,1)$, $a+\alpha\neq 1$ is another typical type that we will treat in this paper.

\item[(iii)] 
We note that  $\phi(t) = t^{-a}\eta(t)$ is
a prototype  in Assumption \ref{Assume-phi}, where
$\eta$ is cut off fuction with $\eta(t)=1$ for $t\in (0, \frac{1}{2})$ and $\eta(t)=0$ for $t\ge \frac{3}{4}$.

 \end{itemize}

As a notational convention, we denote $\rm{sgn}\,(a)=1$ if $a>0$ and $\rm{sgn}\,(a)=-1$ if $a<0$.

Next lemma, some elementary properties of $\Phi$ are stated and, for clarity, its proof will be given in section 3.  
\begin{lemm}\label{lemma240109}
Let  $0 < t < \frac14$ and $\phi$ satisfy the Assumption \ref{Assume-phi}. Then,
for $ 0 <  x_n< \frac14$ and   $  \sqrt{ t} < x_n$,
\begin{align}\label{0110-1}
\phi( t) e^{-\de_1 \frac{x_n^2}{t}} \leq c \phi (x_n^2), \quad {\rm sgn} (\phi')  \phi'( t) e^{-\de_1 \frac{x_n^2}{t}} \leq c \,{\rm sgn} (\phi') \phi' (x_n^2).
\end{align}
In the case $ \phi  = t^{-a} $,
there is a large constant $\ep_1\gg 1$
such that if  $\ep_1 \sqrt{t} < x_n$,  then there exists there a sufficiently small $\ep$ depending on $\ep_1$ such that
\begin{align}\label{240110-5}
t^{-\frac12} e^{-\de_1 \frac{x_n^2}{t}}  \phi( t) \leq \ep  x_n  \,{\rm sgn} (\phi') \phi' (x_n^2).
\end{align}
\end{lemm}


From \eqref{rep-bvp-stokes-w}, \eqref{Poisson-tensor-K}, and    $\eqref{1006-3}_3$,  for $1 \leq i \leq n$,  we have
\begin{align}\label{1220-8}
\begin{split}
  w_i (x,t) & =    4\int_0^t   \int_{\Rn} L_{n i }(x'- y', x_n, t
-s) g_n(y', s) dy' ds\\
  & \quad  + 4 ( 1 - \de_{in}) \int_0^t   \int_{\Rn} B_{i n}(x'- y', x_n, t
-s) g_n(y', s) dy' ds \\
  & \quad   +2    \int_{\Rn} D_{x_i} N (x'- y', x_n) g_n(y', t) dy' \\
&:=   w^{L}_{i}(x,t) + w^{B}_{i}(x,t) + w^{N}_{i}(x,t).
\end{split}
\end{align}

Form now on, for simplicity, we denote  by $\Ga'$  the $n -1$ dimensional Gaussian kernel.
\begin{defin}\label{nota-10}
As a notational convention, we denote
$A\vee B:=\max\bket{A, B}$. Further,  for a positive function $g$, we say that $f\approx \pm e^{-\delta g}$ if $ c^{-1}e^{-g(x)}\le |f(x)|\le ce^{-\frac{1}{8} g(x)}$ for some constant $c>0$.
\end{defin}

\begin{lemm}\label{lemma0709-1}
Let  $\Ga'$ and $N$ be $n -1$ dimensional Gaussian kernel and  $n$ dimensional Newtonian kernel defined in \eqref{H-L-10}.
Let $X'\in\Rn$ with $|X'| \geq 1$,  $ 0  <  x_n < \frac12 $ and $t>0$. Then,
 there exists $c > 0$, independent of $X', x_n$ and $t$, such that
\begin{align}\label{0515-1}
\begin{split}
\int_{\Rn}    \Ga'(X'-z',t)     D_{x_n} N( z',x_n) dz' &=   D_{x_n}  N( X',x_n) + I(X',x_n,t)  + J_1(X', x_n, t),\\
\int_{\Rn}    \Ga'(X'-z',t)    D_{z_i }  N( z',0) dz' &  =   D_{X_i }  N( X', 0)  + J_2(X', t),
\end{split}
\end{align}
where  $i=1,2,\cdots, n-1$ and $I$, $J_1$ and $J_2$ satisfy
\begin{align}\label{Jkl}
\begin{split}
I (x', x_n, t) \approx     t^{-\frac{n-1 }2  }   e^{- \de\frac{|x'|^2}{t}}, \quad
  | J_1(X', x_n,t)|  & \leq     c x_n     t^{\frac{1}{2}  },\qquad
  \abs{J_2 (X', t) }    \leq     c   t^{\frac{1}{2}  } \quad m \geq 1,
  \end{split}
\end{align}
where $''\approx''$ and $\de$ are defined  in Definition \ref{nota-10}.
\end{lemm}
The proof of Lemma \ref{lemma0709-1} will be provided in section \ref{prelim}.

\begin{rem}\label{rem1217}
Let $i=1,2,\cdots, n-1$.
Suppose that $x', y'\in \Rn$ with $|x'| < \frac12$ and $y'\in A$ and, for convenience, we denote $X'=x'-y'$. Then,  it is direct that $\frac{1}{2}<\abs{X'}\le \frac{5}{2}$, because $\frac12  \leq x_i -y_i$ and $  |X'| \leq   (|x'| + | y'|) \leq \frac52$, which implies that
there are positive constants $c_1$ and $c_2$ such that
\[
c_1 \leq  D_{x_i }  N( x' -y', 0)  \leq c_2,\qquad c_1 x_n \leq   D_{x_n}  N( x' -y', x_n)  \leq c_2 x_n.
\]
Therefore,  it follows  from  $\eqref{0515-1}$ and  $\eqref{Jkl}$  that
\begin{align}\label{eq1214-1}
\abs{\int_{\Rn}    \Ga'(x'-y'-z',t)    D_{z_i }  N( z',0) dz' } \leq c, \qquad  |x'| \leq \frac12, \,\,\, y'   \in A, \\
\label{240116-1}
\abs{\int_{\Rn}    \Ga'(x'-y'-z',t)    D_{x_n}  N( z',x_n) dz' } \leq c\Big(  x_n  +  t^{-\frac{n-1 }2  }   e^{-\frac{1}{8t}}  \Big), \qquad  |x'| \leq \frac12, \,\,\, y'   \in A,  
\end{align}
where $ \frac34 < t < \frac54$,  $0<s<t$ and $k\ge 0$.
Furthermore, there exist  $c_0>0$ and $t_0\in (0,  \frac14)$ such that if $0 < t-s < t_0$, then
\begin{align}\label{eq1214-1-1}
c_0\leq   \int_{\Rn}    \Ga'(x'-y'-z',t-s)    D_{z_i }  N( z',0) dz'  \qquad | x'| \leq \frac12, \,\,\, y'   \in A,
\end{align}
\begin{align}
\label{240116-1-8}
c_0 ( x_n+  t^{-\frac{n-1 }2  }   e^{-\frac{1}{t}} ) \leq  \int_{\Rn}    \Ga'(x'-y'-z',t-s)     D_{x_n}  N( z',x_n) dz' , \qquad  |x'| \leq \frac12, \,\,\, y'   \in A.  
\end{align}
Indeed, the estimate \eqref{240116-1-8} is due to the computations that
\[
\int_{\Rn}    \Ga'(x'-y'-z',t-s)     D_{x_n}  N( z',x_n) dz'
\]
\[
\ge D_{x_n}  N( X',x_n)   +  t^{-\frac{n-1 }2  }   e^{-\frac{1}{t}} -\abs{J_1(x', x_n,t-s)}\ge c(x_n +  t^{-\frac{n-1 }2  }   e^{-\frac{1}{t}}  ).
\]
 In conclusion, due to \eqref{eq1214-1}-\eqref{240116-1-8}, we obtain
\begin{equation}\label{March02-10}
 \int_{\Rn}    \Ga'(x'-y'-z',t-s)    D_{z_i }  N( z',0) dz' \approx c, \qquad | x'| \leq \frac12, \,\,\, y'   \in A,
\end{equation}
\begin{equation}\label{March02-20}
\int_{\Rn}    \Ga'(x'-y'-z',t-s)     D_{x_n}  N( z',x_n) dz' \approx x_n+  t^{-\frac{n-1 }2  }   e^{-\frac{1}{t}}, \qquad  |x'| \leq \frac12, \,\,\, y'   \in A. 
\end{equation}
\qed
\end{rem}
We  define for notational convention
\[
\calD_+=\bket{(x,t)\in \mathbb R^n_+ \times  (0, \infty) : |t-1|^\frac12 \ge  x_n,\,\,\abs{t-1}< t_0, \,\,0 < x_n < \frac12},
\]
\[
\calD_-=\bket{(x,t)\in \mathbb R^n_+ \times (0, \infty) : |t-1|^\frac12 < x_n,\,\,\abs{t-1}< t_0,\,\,0 < x_n < \frac12},
\]
\[
\calD_0=\bket{(x,t)\in \mathbb R^n_+ \times (0, \infty) :  \frac{1}{2}< \frac{x_n}{|t-1|^\frac12}<1,\,\,\abs{t-1}< t_0,\,\,0 < x_n < \frac12}.
\]
Here $t_0>0$ is given in Remark \ref{rem1217} and $ t_1 < t_0$ will be determined later.
We also introduce two different regions defined by
\begin{align}\label{regions-10}
\begin{split}
\tilde{\calD}_+&=\calD_+\cap \bket{(x,t):  x_n\leq   \ep_0   | t -1|^\frac12   }, \\
\tilde{\calD}_-&=\calD_-\cap \bket{(x,t):   \ep_1 |t -1|^\frac12\leq x_n },
\end{split}
\end{align}
where $\epsilon_0$ and $\epsilon_1$ be positive constants with $0<\epsilon_0\ll 1\ll \epsilon_1<\infty$, which will be specified later.

\begin{theo}\label{theo1114-1}
Let $w$ be the solution of the Stokes equations \eqref{StokesRn+} in the half-space with boundary data $g$ given in \eqref{0502-6} and \eqref{boundarydata}.  Suppose that $\phi$ satisfies Assumption \ref{Assume-phi}. Then, there is $ 0 < t_1 < t_0$ such that if $\abs{t-1}< t_1 $, $a\neq \frac12$, $ |x'| < \frac12$ and $0 < x_n < \frac12$ such that we prove the following: Let $i=1,2,\cdots, n-1$.
\begin{align}\label{eq1218-10}
\abs{w_i(x,t) }
  \lesssim \left\{\begin{array}{ll} \vspace{2mm}
     1 \quad  &\mbox{if}\,\,\phi(t)=\abs{\ln t}^{-1}\\
    \phi (|1 -t| )
    \quad &\mbox{otherwise}.
    \end{array}
    \right.
\end{align}
In case that $\phi(t)\neq \abs{\ln t}^{-1}$, there exist $\ep_0 $ and $\ep_1$, depending on $t_1 $,  in \eqref{regions-10} such that
\begin{align}\label{eq0828-5-1}
  w_{i} (x,t)
    \ge c\phi ( 1 -t)
    \quad &\mbox{if}\,\,  \,\,   (x,t)\in \tilde{\calD}_- ,\quad t < 1.
\end{align}
\begin{align}\label{eq0828-5-1-1}
  w_{i} (x,t)
    \le -c\phi ( t-1)
    \quad \mbox{if}\,\,  \,\,   (x,t)\in\tilde{\calD}_- \cup \calD_0,\quad  1 < t <1+ t_0.
\end{align}
\begin{align}\label{eq0829-1-8}
 \begin{split}
|D_{x_n} w_{i}(x,t) |   \leq c   \left\{\begin{array}{cl}\vspace{2mm}
   \bke{(t-1) \vee x_n^2}^{ -\frac12} \phi\bke{(t-1) \vee x_n^2}), \quad &\mbox{ if }\,\,1 < t < \frac54,\\
     \bke{(1-t) \vee x_n^2}^{ \frac12} |\phi'\bke{(1-t) \vee x_n^2})|,\quad &\mbox{  if }\,\,\frac34 < t<1.
   \end{array}
   \right.
\end{split}
 \end{align}
 \begin{align}
 \label{1218-20}
D_{x_n} w_i(x,t)
 & \approx
 - ( 1-t)^\frac12  \phi'(1-t)  {\bf I}_{\tilde {\calD_+} }-x_n \phi'(x_n^2)    {\bf I}_{\tilde{\calD_-}}, \quad 1-t_1<t<1.
 \end{align}
\begin{align}\label{0126-1}
D_{x_n} w_i (x,t) \approx -(t-1)^{-\frac12} \phi(t-1)
\quad \mbox{if}\,\,  \,\,   (x,t)\in\tilde{\calD}_+,\quad  1 < t <1+ t_0.
\end{align}
If  $1<t<1+t_1$ and $(x,t)\in\tilde{\calD}_-$
then
\begin{align}\label{Jan02-10}
 \begin{split}
D_{x_n} w_{i}(x,t)  \approx   \left\{\begin{array}{cl}\vspace{2mm}
    - x_n \phi' (x_n^2), \quad &\mbox{ if }\,\,\phi = t^{-a},\\
      - x_n \phi' (x_n^2) + (t -1)^{-\frac32}x_n^2 e^{-\frac{x_n^2}{4(t-1)}}    \phi(t-1 ),\quad &\mbox{ otherwise}.
   \end{array}
   \right.
\end{split}
 \end{align}

\begin{align}\label{1219-3}
 \begin{split}
p(x,t) \approx   \left\{\begin{array}{cl}\vspace{2mm}
   -   \phi' (1-t), \quad &\mbox{ if }\,\,1-t_1<t < 1,\\
    x_n  (t-1)^{-\frac12 } \phi(t-1) \pm 1,\quad &\mbox{ if }\,\, 1 < t< 1+ t_1.
   \end{array}
   \right.
\end{split}
 \end{align}


\end{theo}

\begin{rem}\label{rem0929-1}
The results of Theorem \ref{theo1114-1} imply the following regularity of the solution
for $\abs{t-1}< t_1 $ and $ |x| < \frac12$ near boundary: Let $Q:=B^+_{\frac{1}{2}}\times (1-t_1, 1+t_1)$ and $i=1,2,\cdots, n-1$.
\begin{itemize}
\item[(1)] (The case $ \phi(t)  = |\ln (t)|)$\,\,  For any $p_1<\infty$ and $p_2<3$
\begin{align}
\begin{split}
&w_i \in L^{p_1}(Q), \quad 
   D_{x_n} w_i \in L^{p_2}(Q), \quad  p \in L^1_{weak}(Q)\quad \mbox{ but} \\
& w_i \notin L^\infty(Q),   \quad D_{x_n} w_i \notin L^{3}(Q),\quad p \notin L^1(Q).
\end{split}
\end{align}

\item[(2)] (The case $\phi(t) = |\ln (t)|^{-1}  $)\,\, For any $p_3>3$, $p_4>1$ and $\alpha\in (0, 1)$
\begin{align}
\begin{split}
&w_i \in  L^\infty(Q), \quad  D_{x_n} w_i \in L^{3}(Q), \quad  p \in L^1(Q),  \quad \mbox{ but}\\
&w_i \notin  \calC^{\alpha}(Q), \quad D_{x_n} w_i \notin L^{p_3}(Q),\quad   p \notin L^{p_4}(Q).
\end{split}
\end{align}

 \item[(3)] (The case $\phi(t)= t^{-a}, \,\, 0 < a< 1  $).\,\, For any $p_5<\frac{3}{2a}$,  $p_6<\frac{3}{1+2a}$ and $p_7<\frac{1}{1+a}$
\begin{align}
\begin{split}
&w_i \in L^{p_5}(Q),\quad  D_{x_n} w_i \in L^{p_6}(Q), \quad  p \in L^{p_7}(Q),\quad \mbox{ but}\\
&    w_i \notin L^{\frac3{2a}}(Q), \quad\,D_{x_n} w_i \notin L^{\frac{3}{1+2a}}(Q),\quad \, p \notin L^{\frac1{1 +a} }(Q).
\end{split}
\end{align}

\end{itemize}

\end{rem}

The case that $ \phi(t) = t^{-a}, \,\, 0 < a< 1  $ is most interesting and, combining Theorem \ref{theo1114-1} and Remark \ref{rem0929-1}, we make a table
to show asymptotic behavior of the solution near boundary.
\vspace{5mm}

 \begin{tikzpicture}[scale=3]

\draw (-1.60,  1.90) -- (4.20, 1.90); 
\draw (-1.60,  1.60) -- (4.20, 1.60);  
 \draw (-1.2,  1.20) -- (2.20, 1.20);    
\draw (-1.60,  0.80) -- (4.20, 0.80); 
\draw (-1.20,  0.40) -- (4.20, 0.40); 
\draw (-1.20,  0.00) -- (4.20, 0.00); 
\draw (-1.60,  -0.40) -- (4.20, -0.40); 
\draw (-1.60,  -1.00) -- (4.20, -1.00); 

\draw(-1.60, 1.90) --(-1.60, -1.0); 
\draw( -1.2, 1.90 ) -- (-1.2, -0.40 ); 

\draw(-0.80, 1.90)--(-0.80, -1.0); 
\draw(0.10, 1.90) --(0.10, -1.0); 
\draw(2.20, 1.90) --(2.20, -1.0); 
\draw(4.20, 1.90)--(4.20,  -1.0); 

\draw (-0.990, 1.750) node {\small{Region}};
\draw (-0.3, 1.730) node {\small{$w_i$}};
\draw (1.15, 1.730) node {\small{$D_{x_n} w_i$}};
\draw(3.2, 1.74) node {\small{$ p$}};

\draw(-1.40, 1.2) node {$t < 1$};
\draw(-1.40, 0.2) node {$t > 1$};

\draw(-0.99, 1.40) node {$  \small{\tilde D_+}$};
\draw(-0.990, 1.0) node {$  \small{\tilde D_-}$};

\draw(-0.99, 0.60) node {$  \small{\tilde D_+}$};
\draw(-0.990, 0.20) node {$  \small{\tilde D_-}$};
\draw(-0.990, -0.20) node {$  \small{\tilde D_0}$};

\draw (1.15, 1.4) node {\small{$ \approx - ( 1-t)^{\frac12}  \phi' (1-t) \approx (1-t)^{-\frac12 -a }$}};

  \draw(-0.38, 1.0) node { \small{ $\approx \phi (1-t) $}};
  \draw (1.15, 1.0) node {\small{$ \approx -x_n \phi'(x_n^2) \approx x_n^{-1 -2a}$} };
  \draw (3.2, 1.2) node {\small{$ \approx -\phi' (1-t) \approx (1 -t)^{-1 -a}$} };

   \draw (1.15, 0.60) node {\small{$ \approx -(t-1)^{-\frac12} \phi(t-1) \approx -(t-1)^{-\frac12-a}$} };

   \draw (-0.35, 0.220) node {\small{$ \approx   \phi(t-1) $} };
      \draw (1.15, 0.20) node {\small{$ \approx -x_n \phi'(x_n^2)\approx x_n^{-1 -2a}$} };
            \draw (3.2, 0.20) node {\small{$ \approx -x_n(t-1)^{-\frac12 } \phi(t-1) \approx (t -1)^{-a} $} };

   \draw (-0.35, -0.220) node {\small{$ \approx   \phi(t-1) $} };
              \draw (3.2, -0.20) node {\small{$ \approx -x_n(t-1)^{-\frac12 } \phi(t-1) \approx (t -1)^{-a} $} };

   \draw (-1.20, -0.620) node {\small{ Integrable} };
      \draw (-1.2, -0.820) node {\small{ property} };
      
   \draw (-0.35, -0.620) node {\small{$\in L^p, \,\, p < \frac3{2a} $} };
      \draw (-0.35, -0.820) node {\small{$\mbox{but} \,\, \notin L^{\frac3{2a}}  $} };
      
  \draw (1.15,  -0.720) node {\small{$\in L^p, \,\, p < \frac3{1 + 2a}  \quad \mbox{but} \,\, \notin L^{\frac3{1 +2a}}$} };
   \draw (3.2,  -0.720) node {\small{$\in L^p, \,\, p < \frac1{1 + a}  \quad \mbox{but} \,\, \notin L^{\frac1{1 +a}}$} };
   
   \draw (1.2, -1.2) node {$ \big( \phi (t) = t^{-a}, \,\, 0 < a < 1 \big)$};
\end{tikzpicture}
\vspace{0.5mm}

\begin{rem}
When it comes to global estimate of the solution $w$ in the half-space, 
it is worth mentioning that $w$ belongs to the same spaces as shown in Remark \ref{rem0929-1}. For example, in the case $\phi= |\ln t|^{-1}$,  since the boundary data $g_{n}(t)$ in  \eqref{0502-6} is Dini-continuity for all $ 0 < t < \infty$, it follows due to \cite{CC1} that $w$ is boudned in the half space, i.e. 
$w \in L^\infty (\R_+ \times (0, \infty))$ (see also \cite{CC2} and \cite{CCK}).
Moreover, since $\phi\in L^{\infty}\cap \dot B^{\frac{1}3}_{3}$, using the results in \cite{CJ1} and \cite{CJ2}, one can see that
\begin{align}\label{0402-10}
\| \nabla w\|_{L^{3}(\R_+ \times (0,\infty)) } &  \leq c \bke{\| \phi \|_{L^3 (0, \infty)} +   \| \phi  \|_{\dot B^{\frac{1}3}_{3} (0,\infty)}}< \infty.
\end{align}
The other cases such as $\phi(t)=\abs{\ln t}$ or $t^{-a}$ can be treated similarly, and thus we skip their details.
Our results can be understood that temporal singular behaviors caused by compactly supported boundary data are immediately transfered to all over the place and, in addition, unlike the interior  local smoothing in spatial variables is not available near boundary. We remark that spatial singular behavior of compactly supported boundary data is not delivered to other places and local smoothing effect is available in any region away from the data (see \cite{KM24}).
\end{rem}

\begin{rem}
We remark that authors showed in \cite{CK20} that there is a very weak solution of \eqref{Stokes-10} or \eqref{nse-30} so that $\nabla w $ is not squre integrable, i. e.,
$\norm{\nabla w}_{L^{2}(Q_1^+)}=\infty$, although
$\norm{w}_{L^2(Q_1^+)}<\infty$.
It turns out that it is a special case  almost close to $a=\frac{1}{4}$ and thus it is observable that $  w$ is not bounded and $\pi$ is not  integrable in $Q_1^+$.
\end{rem}

Next, we will establish existence of singular solutions of the Navier-Stokes equations as in the Stokes system. When it comes to the boundary data for temporal variable, we consider only a specific case in the Assumption \ref{Assume-phi}, which satisfies the following:
\begin{equation}\label{prototype-phi}
\phi(t)=t^{-a},\qquad t\in (0, \frac{1}{2}],\quad a\in (0,1),
\end{equation}
since we think that it is most interesting. The other cases such as $\phi(t)=\abs{\ln s}$ or $\phi(s)=\abs{\ln s}^{-1}$ could be treated similarly, and thus we omit the details of those cases (see Remrak \ref{NSE-log}).

Let $\varphi_1 \in C_c^\infty({\mathbb R}^{n})$ be a  cut-off function satisfying $ \varphi_1 \geq 0$,  ${\rm supp} \, \varphi_1 \subset B_\frac1{\sqrt{2}}$ and $\varphi_1 \equiv 1 $ in $B_{\frac38}$. Also, let $\varphi_2 \in C_c^\infty(-\infty, \infty)$ be a  cut-off function satisfying $ \varphi_2 \geq 0$,  ${\rm supp} \, \varphi_2 \subset (-3, 3)$ and $\varphi_2 \equiv 1 $ in $(-2, 2)$. Let $\varphi (x,t) =\varphi_1 (x) \varphi_2 (t)$.
Let $U = w\varphi$ and $\Pi = \pi \varphi$ such that $U|_{Q^+_\frac14} =w$ and $\Pi|_{Q^+_\frac14} =\pi$.

We consider the following perturbed Navier-Stokes equations in
$\R_+\times (0,\infty)$:
\begin{equation}\label{CCK-Feb7-10}
v_t-\Delta v+\nabla q+{\rm div}\,\left(v\otimes v+v\otimes
U+U\otimes v\right)=-{\rm div}\,(U\otimes U), \quad {\rm
div} \, v =0
\end{equation}
with homogeneous initial and boundary data, i.e.
\begin{equation}\label{pnse-bdata-20}
v(x,0)=0\,\,\mbox{in}\,\,\mathbb R^n_+,\qquad v(x,t)=0 \,\,\mbox{on} \,\,\{x_n=0\}\times [0, \infty).
\end{equation}

\begin{prop}\label{theo1001}
Suppose that $\phi$ satisfies Assumption \ref{Assume-phi} and \eqref{prototype-phi}.
Let $ 0 < a < \frac3{2(n+2)}$ and $ \frac{n+2}2 < q < \frac3{4a}$. 
Then, we have $ \| U\|_{L^{2q} } < c\al$, where $\al > 0$ is defined in \eqref{0502-6}.
If $\alpha$ is taken to be sufficiently small, then there is a solution $ v$  of \eqref{CCK-Feb7-10}-\eqref{pnse-bdata-20} such that
\begin{align}\label{pNSE-100}
\| \na v\|_{L^q (\R_+ \times (0, 1))}  + \|  v\|_{L^{2q} (\R_+ \times (0, 1))} \leq c (\| w\|_{L^q(Q^+_2)},  \| w\|_{L^{2q}(Q^+_2)} ).
\end{align}
In case that $ 0 < a < \frac3{4(n+2)}$, we obtain  $ \|  w\|_{L^{2q}}  < c\al$ for $ n+2 < q < \frac3{4a}$. 
If $\alpha$ is taken to be sufficiently small, then
$v$ satisfies
\begin{align}\label{pNSE-200}
 \|  v\|_{L^{\infty} (\R_+ \times (0, 1))}\leq c ( \| w\|_{L^{q}(Q^+_2) }, \|  w\|_{L^{2q}(Q^+_2) } ) ).
\end{align}
\end{prop}

The proof of Proposition \ref{theo1001} will be given in section \ref{NSequations}.

We then set $u:=v+U$ and $p =\pi + q$, which becomes a weak solution of the
Navier-Stokes equations in $\tilde{Q}_+=B_{\frac{1}{2}}^+\times (0,2)$, namely
\begin{equation}\label{NSE-March5}
u_t-\Delta u+\nabla p=-{\rm div}\,\left(u\otimes u\right),\qquad
{\rm div} \, u=0 \quad Q^+_1,
\end{equation}
with boundary data $u(x,t)=0$ on $\{x_n=0\}$ and homogeneous
initial data $u(x,0)=0$.

With the aid of Proposition \ref{theo1001}, we can construct a solution of the Navier-Stokes equations such that 
its behaviors near boundary  are more or less the same as those of the Stokes equations mentioned in Theorem \ref{theo1114-1}. Since its verification is straightforward by combining Theorem \ref{theo1114-1} and Proposition \ref{theo1001}, we just state it without providing its details.

\begin{theo}\label{mainthm-NSE}
Suppose that $\phi$ satisfies Assumption \ref{Assume-phi} and and \eqref{prototype-phi} with an additional restriction that $ 0 < a <   \frac{3}{2(n+2)}$ for $\phi(t)=t^{-a}$. Then, there exists 
a solution of the Navier-Stokes equations, $u$, in  \eqref{NSE-March5} such that it satisfies the same regularity properties mentioned in Remark \ref{rem0929-1}.
Furthermore, if $ 0 < a < \frac3{4(n+2)}$, then $u$ satisfies pointwise estimates \eqref{eq1218-10}, \eqref{eq0828-5-1} and  \eqref{eq0828-5-1-1}.
\end{theo}

\begin{rem}
The pointwise estimate of $u$ in Theorem \ref{mainthm-NSE} is the combination of \eqref{pNSE-200} for $v$ in Proposition \ref{theo1001} and \eqref{eq1218-10}-\eqref{eq0828-5-1-1} for $w$ in Theorem \ref{theo1114-1}. However, 
the pointwise estimate of $\nabla u$ is not clear because of the lack of that of $\nabla v$, and instead, $L^q$ estimate is only available up to $q<\frac{3}{4a}$ as shown in \eqref{pNSE-100}. Thus, we leave it as an open question on construction of solutions of the Navier-Stokes equations whose normal derivatives satisfies pointwise estimates near boundary given in Theorem \ref{theo1114-1}. 
\end{rem}

\begin{rem}\label{NSE-log}
In the case that $\phi(t)=\abs{\ln s}$ or $\phi(s)=\abs{\ln s}^{-1}$, the estimate \eqref{pNSE-100} is available for any expoent $q<\infty$. Thus, it is also available to construct
a singular solution of the Navier-Stokes equations with the same regularity properties as the Stokes system mentioned in Remark \ref{rem0929-1}.
\end{rem}

This paper is organized as follows: In section \ref{prelim}, proofs of some lemmas are provided.
Section \ref{pf-SS} is devoted to proving Theorem \ref{theo1114-1}.
We present the proof of Proposition \ref{theo1001} in Section \ref{NSequations}.


\section{Preliminaries}
\label{prelim}
\setcounter{equation}{0}

In this section, we mainly provide proofs of lemmas, which are useful for our analysis.
Before we do that, we first define the notion of weak solutions of the Stokes equations in the half space.
\begin{defin}\label{stokes-def}
Let $\phi$ satisfy Assumption \ref{Assume-phi}.
We say that $w\in L^1_{\rm{loc}}(0, \infty; W^{1,1}_{\rm{loc}}({\mathbb R}^n_+))$ is a weak solution of the Stokes equations \eqref{StokesRn+} in the half-space with boundary data $g$ given in \eqref{0502-6} and \eqref{boundarydata}, if the following equality holds:
\begin{align*}
\int^T_0\int_{\hR}w\cdot \Phi_t dxdt+\int^T_0\int_{\hR }w\cdot \Delta \Phi
dxdt+\int^T_0 \int_{\pa \hR} g_n \cdot D_{x_n}\Phi  dx' dt=0
\end{align*}
for each $\Phi\in C^2_c(\overline{\hR}\times [0,T])$ with
\begin{align}\label{0528-1}
\mbox{\rm div } \Phi=0,\quad \Phi\big|_{\pa \hR \times (0, T)}=0, \quad
\Phi(\cdot,T) =0.
\end{align}
In addition, for each $\Psi\in C^1_c(\overline{\hR})$ with
 $\Psi \big|_{\pa \hR }=0$
\begin{equation}\label{Stokes-bvp-2200}
\int_{\hR} u(x,t) \cdot \na \Psi(x) dx =0  \quad  \mbox{ for all}
\quad 0 < t< T.
\end{equation}
\end{defin}

For notational convenition, we denote for  fixed $\al \geq 0$ and $k=0,1$
\begin{align}\label{notation-10}
\Psi_{\phi, k}(t):&=
\int_0^{\min(1,t)}( t-s)^{-\al}  \phi^{(k)}(1-s) ds,
\\
\label{notation-20}
{\tilde{\Psi}}_{\phi, k}(t):
&=\int_0^{\min(1,t)}  (t-s)^{-\alpha} e^{- \frac{x_n^2}{4(t-s)}}   \phi^{(k)}(1-s) ds. 
\end{align}
The following is a key lemma to obtain main results.
\begin{lemm}\label{key-lemma}
Let  $ 0  <  \al$, $\al \neq 1$, $\alpha+a\neq 1$
 and  $\frac34 < t < \frac54$.
Suppose that $\phi$ satisfies Assumption \ref{Assume-phi}.
Then, $\Psi_{\phi, k}(t)$ and ${\tilde\Psi}_{\phi, k}(t)$ in \eqref{notation-10}-\eqref{notation-20} sastisfy following:
\begin{itemize}
\item[(i)] Let $1<t<\frac54$.  Then,
\begin{align}
\label{phi-5}
\Psi_{\phi,0}(t)
\approx\left\{\begin{array}{cl}\vspace{2mm}
 (t-1)^{ 1-\al} \phi(t-1)
 \quad &\mbox{ if }\,\,\al+a>1,\\
 1 \quad &\mbox{ otherwise }.
 \end{array}
 \right.
\end{align}
\begin{equation}\label{March02-100-10}
\Psi_{\phi,1}(t)\approx (t-1)^{ 1-\al} \phi'(t-1).
\end{equation}
\begin{align}
\label{tildephi-15} &  {\tilde\Psi}_{\phi, 0}(t)
\approx\left\{\begin{array}{cl}\vspace{2mm}
\bke{(t-1) \vee x_n^2}^{1-\alpha } \phi\bke{(t-1) \vee x_n^2})
 \quad &\mbox{ if }\,\,\al+a>1,\\
 1 \quad &\mbox{ otherwise }.
 \end{array}
 \right.
\end{align}

\begin{align}
\label{tildephi-10-1} {\tilde\Psi}_{\phi, 1}(t) &\approx
((t-1)\vee x_n^2)^{1-\alpha } \phi'( (t-1)\vee x_n^2).
\end{align}

\item[(ii)] Let $\frac34<t<1$.  Suppose that $0<\alpha<1$. Then,
\begin{align}
\label{phi-50}
\Psi_{\phi,0}(t)
\approx\left\{\begin{array}{cl}\vspace{2mm}
 (1 -t)^{ 1-\al} \phi(1-t)
 \quad &\mbox{ if }\,\,\al+a>1, \\
 1 \quad &\mbox{ otherwise }.
 \end{array}
 \right.
\end{align}

\begin{equation}\label{March02-100-20}
\Psi_{\phi,1}(t)\approx (1-t)^{ 1-\al} \phi'(1-t).
\end{equation}

\begin{align}
\label{tildephi-25} &  {\tilde\Psi}_{\phi, 0}(t)
 \approx\left\{\begin{array}{cl}\vspace{2mm}
 \max(1,  x_n^{-2\alpha +2} \phi(x_n^2))  + \phi( 1-t)  (1-t)^{2-\al} x_n^{ -2}
 e^{-\de \frac{x_n^2}{1-t}} \quad &\mbox{ if }\,\, 1-t< x_n^2, \quad \al  > 0,\\
 \vspace{2mm}
 \phi( 1-t)  x_n^{2 -2\alpha} \quad &\mbox{ if }\,\, x_n^2 < 1-t,   \quad \al > 1,\\
 \max(1, (1 -t)^{-\alpha +1} \phi(1-t))  \quad &\mbox{ if }\,\, x_n^2 < 1-t,    \quad  \al < 1.
 \end{array}
 \right.
\end{align}

\begin{align}
\label{tildephi-25-2} &  {\tilde\Psi}_{\phi, 1}(t)
 \approx\left\{\begin{array}{cl}\vspace{2mm}
 x_n^{-2\alpha +2} \phi' (x_n^2) +\phi'( 1-t)  (1-t)^{2-\al} x_n^{ -2}
 e^{-\de \frac{x_n^2}{1-t}}   \quad &\mbox{ if }\,\, 1-t< x_n^2, \quad \al  > 0,\\
 \vspace{2mm}
 \phi'( 1-t)  x_n^{2 -2\alpha} \quad &\mbox{ if }\,\,x_n^2 < 1-t,   \quad \al > 1,\\
 (1 -t)^{-\alpha +1} \phi'(1-t)  \quad &\mbox{ if }\,\, x_n^2 < 1-t,    \quad  \al < 1.
 \end{array}
 \right.
\end{align}

\end{itemize}

\end{lemm}

\begin{proof}
We first treat the cases $\Psi_{\phi,0}$ and $\tilde \Psi_{\phi,0}$.
We first remind that
\begin{align}\label{eq1113-1}\begin{split}
&\frac12 ( 1-s) \leq  t-s \leq  (1-s), \qquad  \frac34 < t < 1 \qquad \mbox{for} \quad    0< s < 2t -1, \\
&\frac12( 1-s) \leq 1-t \leq  (1-s), \qquad  \frac34 < t < 1 \qquad \mbox{for} \quad    2t-1< s <  t, \\
&  ( 1-s) \leq  t-s \leq  2(1-s), \qquad  1 < t<\frac54 \qquad \mbox{for} \quad   0 < s < 2-t, \\
& \frac12( t-s) \leq t-1 \leq  (t-s), \qquad  1 <t < \frac54 \qquad \mbox{for} \quad    2-t< s <  1.
\end{split}
\end{align}

$\bullet$ {(The case $ 1 <  t  <\frac54$)}\,\,
We split the integral over $[0,1]$ into $[0, 2-t]$ and $[2-t , 1]$. Then, from $\eqref{eq1113-1}_{3,4}$ 
and the change of variables, we have
\begin{align}\label{eq1017-1}
\begin{split}
 \Psi_{\phi,0}(t)
 &  \approx      \int_0^{2-t}   (1 -s)^{-\al} \phi( 1-s) ds + \int_{2-t}^1 (t -1)^{-\al} \phi( 1-s) ds,  \\
&=    \int^{1}_{t-1}     s^{-\al} \phi ( s) ds + (t -1)^{-\al}\int_0^{t-1}  \phi( s) ds.
\end{split}
\end{align}
Applying \eqref{phi-assume-v15} and   \eqref{phi-assume-v17} in \eqref{eq1017-1},   we obtain
\eqref{phi-5}.

Using $\eqref{eq1113-1}_4$, we have  $e^{-\frac{x_n^2}{2(t-1)}} \leq e^{-\frac{x_n^2}{2(t-s)}} \leq e^{-\frac{2x_n^2}{(t-1)}} $ for $2 -t < s < 1$, namely
$e^{-\frac{x_n^2}{4(t-1)}} \leq e^{-\frac{x_n^2}{4(t-s)}}\leq e^{-\frac{x_n^2}{(t-1)}}$. Thus,
we note that
 $e^{-\frac{x_n^2}{4(t-s)}} \approx  e^{-\de \frac{x_n^2}{t-1}}$ for $2 -t < s < 1$ (see Notation  \ref{nota-10}).
Similarly, we split the following integral:
\begin{align}\label{eq1016-1}
\begin{split}
{\tilde\Psi}_{\phi,0}(t)
& \approx 
\int_{t-1}^1 s^{-\alpha} e^{-\frac{x_n^2}{4s}}  \phi( s) ds+   (t-1)^{-\alpha} e^{-\de_1\frac{x_n^2}{t-1}}  \int_0^{t-1}\phi( s) ds.
\end{split}
\end{align}
Due to  \eqref{phi-assume-v17},  we have
\begin{align}
\label{eq1016-2}
\begin{split}
  (t-1)^{-\alpha} e^{-\de_1 \frac{x_n^2}{t-1}} \int_0^{t-1} \phi  (s) ds
& \approx  (t -1)^{ -\alpha  +1} \phi (t-1)    e^{- \de_1 \frac{x_n^2}{t-1}}.
\end{split}
\end{align}
For the first term of the right-hand side of  \eqref{eq1016-1}, we  separate the cases that $x_n^2 \le t-1$  and $x_n^2 > t-1$.
In the case that $ x_n^2 \leq t-1$, since $e^{-\frac{x_n^2}{4s}} \approx 1$ for $ x_n^2 \leq  t -1  \leq s  <1$, we can see via \eqref{phi-assume-v15} that
\begin{align}
  \label{eq1016-3}
\int_{t-1}^1 s^{-\alpha} e^{-\frac{x_n^2}{4s}}  \phi( s) ds& \approx
 \int^{1}_{t-1}     s^{-\al} \phi ( s) ds
  \approx     \max (1, ( t-1)^{-\al +1} \phi(t-1)).
 \end{align}
Thus, if $ x_n^2 \leq t-1$, it follows from \eqref{eq1016-1},  \eqref{eq1016-2} and \eqref{eq1016-3} that
\begin{align}\label{tildephi-zero-20}
\begin{split}
{\tilde\Psi}_{\phi,0}(t)
  \approx    \max (1, ( t-1)^{-\al +1} \phi(t-1)).
 \end{split}
 \end{align}
Secondly, we treat the case that $ t-1 < x_n^2$. As in \eqref{eq1016-3}, we have
\begin{align}\label{eq1016-5}
\begin{split}
  \int_{t-1}^1 s^{ -\alpha} e^{-\frac{x_n^2}{4s}}  \phi( s)  ds
 & =   \int_{t-1}^{x_n^2} s^{-\alpha}   e^{-\frac{x_n^2}{4s}}    \phi(s) ds +  \int_{x_n^2}^1 s^{-\alpha}   e^{-\frac{x_n^2}{4s}}  \phi(s) ds.
  \end{split}
\end{align}
Since $ e^{-\frac{x_n^2}{4s}} \approx 1$ for $ x_n^2 < s<1 $,   it follows from    \eqref{phi-assume-v15} that
\begin{align}\label{eq1016-6}
\begin{split}
\int_{x_n^2}^1 s^{-\alpha}   e^{-\frac{x_n^2}{4s}}  \phi(s) ds \approx \int^1_{x_n^2} s^{-\alpha}  \phi(s) ds
\approx  \max(1,  x_n^{-2\alpha +2} \phi(x_n^2)).
\end{split}
\end{align}
This implies that
\begin{align}\label{eq1016-6-5}
 \max(1,  x_n^{-2\alpha +2} \phi(x_n^2))\lesssim \int_{t-1}^1 s^{ -\alpha}   e^{-\frac{x_n^2}{4s}} \phi( s)  ds.
\end{align}

On the other hand, noting that  $e^{-b} \leq c_m b^{-m}$ for any $ b > 0$ and $m \in {\mathbb N}$, we have
\begin{align}
\label{eq1016-7}
\begin{split}
\int_{t-1}^{x_n^2} s^{-\alpha}  e^{-\frac{x_n^2}{4s}}  \phi(s) ds
&  \leq  c\int_{t-1}^{x_n^2} s^{-\alpha}  (\frac{x_n^2}{4s})^{-m} \phi(s) ds
  \leq  c x_n^{-2m} \int_{t-1}^{x_n^2} s^{-\alpha + m } \phi(s) ds\\
 &  \leq  c x_n^{-2m}  x_n^{-2\alpha + 2m +2} \phi(x_n^2)
   \leq  c   x_n^{-2\alpha   +2} \phi(x_n^2).
 \end{split}
 \end{align}
Estimates \eqref{eq1016-5} -\eqref{eq1016-7} implies that   for  $ 1 < t  $ and $ t-1 < x_n^2<\frac12$
\begin{align}\label{eq1016-8}
 \int_{t-1}^1 s^{ -\alpha}     e^{-\frac{x_n^2}{4s}}  \phi(s) ds
 & \approx \max(1,  x_n^{-2\alpha +2} \phi(x_n^2)).
\end{align}
Adding up \eqref{eq1016-1},  \eqref{eq1016-2}, \eqref{tildephi-zero-20}  and \eqref{eq1016-8}, we get  \eqref{tildephi-15}.
\\
\\
$\bullet$ (The case $ \frac34 < t < 1$)\,\,
Noting that  $ \frac12 < 2t -1 < t$ and using $\eqref{eq1113-1}_1$,   \eqref{phi-assume-v11}, \eqref{phi-assume-v12}, \eqref{phi-assume-v15} 
and change of variables, it follows that
\begin{align}\label{eq1017-9}
\begin{split}
 \Psi_{\phi,0}(t)
& \approx    \int_0^{2t -1}   (1 -s)^{-\al} \phi( 1-s) ds + \int_{2t -1}^t    (t -s)^{-\al} \phi( 1-t) ds\\
&  =
\int^1_{2-2t}  s^{-\al}  \phi( s) ds  +   \phi (  1-t) \int_0^{1-t} s^{-\al}   ds\\
&        \approx \max( 1,   ( 1-t)^{1 -\al} \phi(1-t)).
\end{split}
\end{align}
Hence, we obtain  \eqref{phi-50}.
It remains to show \eqref{tildephi-25}.
We note first via  $\eqref{eq1113-1}_{1,2}$, \eqref{phi-assume-v11} and \eqref{phi-assume-v12} that
\begin{align}
\label{eq1017-10}
\begin{split}
{\tilde\Psi}_{\phi,0}(t)
&\approx \int_0^{2t -1}  (1-s)^{-\alpha} e^{-\de_1 \frac{x_n^2}{1-s}}  \phi ( 1-s)ds  + \int_{2t-1}^t (t-s)^{-\alpha} e^{-\frac{x_n^2}{4(t-s)}}  \phi ( 1-t)ds\\
&:=I_1+I_2.
\end{split}
  \end{align}
In case that $ x_n^2  <   1-t$, since $ e^{-\de_1\frac{x_n^2}s} \approx 1$ for $ 1-t  <  s <1 $, it follows due to \eqref{phi-assume-v15} that
\begin{align}\label{eq1016-11-1}
\begin{split}
I_1  =  \int_{2-2t}^1 s^{-\alpha}   e^{-\de_1\frac{x_n^2}{s}} \phi(s) ds
&\approx   \int_{2-2t}^1 s^{-\alpha}  \phi(s) ds
    \approx  \max(1,
(1 -t)^{-\alpha +1} \phi(1-t)).
   \end{split}
\end{align}
For $I_2$, using the change of variables, we have
\begin{align}\label{eq1016-9-1-1-1}
\begin{split}
I_2  &=  \phi ( 1-t)  \int_0^{1-t}s^{-\alpha} e^{-\frac{x_n^2}{4s}} ds
  =  \phi  (1 -t) x_n^{2 -2\alpha}\int^\infty_{\frac{x_n^2}{1-t}}  s^{-2 + \alpha}  e^{-\frac{s}4}  ds \\
& \approx  \left\{\begin{array}{cl} \vspace{2mm}
\phi( 1-t)  x_n^{2 -2\alpha}    \quad &\mbox{ if }\,\,\al > 1,\\
 \phi( 1-t)  ( 1-t)^{1 -\al}   \quad  &\mbox{ if }\,\,\al < 1.
 \end{array}
 \right.
  \end{split}
  \end{align}
Summing up \eqref{eq1017-10}, \eqref{eq1016-11-1} and    \eqref{eq1016-9-1-1-1}, we obtain
\begin{align*}
{\tilde\Psi}_{\phi,0}(t)  & \approx   \left\{\begin{array}{cl} \vspace{2mm}
 \phi( 1-t)  x_n^{2 -2\alpha}  +  \max(1, (1 -t)^{-\alpha +1} \phi(1-t))    \quad &\mbox{ if }\,\,\al > 1,\\
 \max(1, (1 -t)^{-\alpha +1} \phi(1-t))    \quad  &\mbox{ if }\,\,\al < 1
 \end{array}
 \right.\\
 & \approx   \left\{\begin{array}{cl} \vspace{2mm}
 \phi( 1-t)  x_n^{2 -2\alpha}    \quad &\mbox{ if }\,\,\al > 1,\\
 \max(1, (1 -t)^{-\alpha +1} \phi(1-t))    \quad  &\mbox{ if }\,\,\al < 1.
 \end{array}
 \right.
\end{align*}
Thus, we have \eqref{tildephi-25} for the case that $x_n^2 < 1-t$.

On the other hand, in case that $ 1-t < x_n^2$, we split $I_1$ as follows:
\begin{align}\label{eq1016-12-1}
I_1
  &  =   \int_{2-2t}^{x_n^2} s^{-\alpha} e^{-\de_1 \frac{x_n^2}{s}} \phi(s) ds +  \int_{x_n^2}^1 s^{-\alpha} e^{-\de_1 \frac{x_n^2}{s}} \phi(s) ds:=I_{11}+I_{12}.
\end{align}
It is straightforwad from \eqref{phi-assume-v15} that
\begin{align}\label{eq1018-1-1}
\begin{split}
I_{12}
\approx  \int^1_{x_n^2} s^{-\alpha}  \phi(s) ds
 \approx
\max(1,  x_n^{-2\alpha +2} \phi(x_n^2)).
   \end{split}
   \end{align}
 With the same estimate of \eqref{eq1016-7}, we get
   \begin{align}
\label{eq1018-2-1}
\begin{split}
I_{11} = \int_{2-2t}^{x_n^2} s^{-\alpha} e^{- \de_1\frac{x_n^2}{s}} \phi(s) ds
 &  \leq  c   x_n^{-2\alpha  +2} \phi(x_n^2).
 \end{split}
 \end{align}
Adding up  \eqref{eq1016-12-1}, \eqref{eq1018-1-1} and \eqref{eq1018-2-1}, we obtain   for   $ 1-t < x_n^2$
\begin{align}\label{eq0117-70}
I_1 = \int_0^{2t -1}  (1-s)^{-\alpha} e^{-\de_1 \frac{x_n^2}{1-s}}  \phi ( 1-s)ds
\approx \max(1,  x_n^{-2\alpha +2} \phi(x_n^2)).
\end{align}
Direct computations via the change of variables give
\begin{align}\label{eq1016-9}
\begin{split}
I_2  &  =  \phi  (1 -t) x_n^{2 -2\alpha}\int^\infty_{\frac{x_n^2}{1-t}}  s^{-2 + \alpha}  e^{-\frac{s}4}  ds \approx \phi( 1-t)   (1-t)^{2-\al} x_n^{-2}
 e^{-\de  \frac{x_n^2}{1-t}}.
  \end{split}
  \end{align}
Hence, from \eqref{eq1017-10},  \eqref{eq1016-9},  \eqref{eq0117-70} and Lemma \ref{lemma240109}, for $ 1 -t \leq  x_n^2 $,
\begin{align*}
{\tilde\Psi}_{\phi,0}(t)  & \approx  \max(1,  x_n^{-2\alpha +2} \phi(x_n^2)) +   \phi( 1-t)   (1-t)^{2-\al} x_n^{-2}
 e^{- \de \frac{x_n^2}{1-t}}.
 \end{align*}
We obtain \eqref{tildephi-25}.
\\
\\
Next, we prove the estimate \eqref{tildephi-10-1} of  ${\tilde{\Psi}}_{\phi, 1}(t)$.
Again we separately consdier the cases that $t<1$ and $t>1$.
\\
$\bullet$ {(The case $ 1 <  t  <\frac54$)}\,\,
For convenience, we define 
\begin{align}
\label{notation-30}
{\tilde{\Psi}}_{\phi, k}(t;\tau):=
&\int_0^{\tau}  (t-s)^{-\alpha} e^{- \frac{x_n^2}{4(t-s)}}   \phi^{(k)}(1-s) ds. \qquad k=0,1.
\end{align}
 We then split the integral, ${\tilde\Psi}_{\phi,1}(t;2-t)$, as follows:
\begin{align}  \label{240105-3}
\begin{split}
{\tilde\Psi}_{\phi,1}(t;2-t)    
& =  \int_0^{\frac34}(t-s)^{-\alpha} e^{-\frac{x_n^2}{4(t-s)}}  \phi'( 1-s) ds    +  \int_\frac34^{2-t}(t-s)^{-\alpha} e^{-\frac{x_n^2}{4(t-s)}}  \phi'( 1-s) ds.
\end{split}
\end{align}
Since $\phi$ is smooth in $[\frac12,1]$, from direct caculation, we have for $0 < x_n < \frac12$
\begin{align}\label{240105-1}
 |\int_0^{\frac34}(t-s)^{-\alpha} e^{-\frac{x_n^2}{4(t-s)}}  \phi'( 1-s) ds  | \leq A_0< \infty. 
\end{align}
For the second term in \eqref{240105-3}, using
\eqref{phi-assume-v13} and \eqref{phi-assume-v14}, we have
\begin{align}  \label{240105-2}
\int_\frac14^{2-t}(t-s)^{-\alpha} e^{-\frac{x_n^2}{4(t-s)}}  \phi'( 1-s)     
&\approx \int_\frac14^{2-t}(1-s)^{-\alpha} e^{- \de_1\frac{x_n^2}{1-s}}  \phi'( 1-s)     =  \int_{t-1}^\frac14 s^{-\alpha} e^{-\de_1 \frac{x_n^2}s}  \phi'( s).
\end{align}
In the case that $ x_n^2 \leq t-1$, since $e^{- \de_1 \frac{x_n^2}s} \approx 1$ for $ x_n^2 \leq  t -1  \leq s  <1$, we can see via \eqref{phi-assume-v16} that
\begin{align} \label{eq1016-1-1}
\begin{split}
 \int_{t-1}^\frac14 s^{-\alpha} e^{- \de_1 \frac{x_n^2}s}  \phi'( s) ds \approx  \int_{t-1}^\frac14  s^{-\alpha}    \phi'( s) ds
  \approx      ( t-1)^{-\al +1} \phi'(t-1).
\end{split}
\end{align}
Adding up  \eqref{240105-3}, \eqref{240105-1}, \eqref{240105-2} and \eqref{eq1016-1-1},  it follows that
\begin{align}\label{tildephi-zero-20-1}
\begin{split}
{\tilde\Psi}_{\phi,1}(t;2-t)    \approx      ( t-1)^{-\al +1} \phi'(t-1),
 \end{split}
 \end{align}
and thus we obtain \eqref{tildephi-10-1}  for $x_n^2 < t-1$.

Secondly, if  $ t-1 < x_n^2$,  we split the integral \eqref{eq1016-1-1} as follows:
\begin{align}\label{eq1016-5-1}
\begin{split}
  \int_{t-1}^\frac14  s^{ -\alpha}  e^{-\de_1 \frac{x_n^2}s} \phi'( s)  ds
 & =   \int_{t-1}^{x_n^2} s^{-\alpha}e^{- \de_1 \frac{x_n^2}s}  \phi'(s) ds +  \int_{x_n^2}^\frac14 s^{-\alpha}e^{-\de_1 \frac{x_n^2}s} \phi'(s) ds.
  \end{split}
\end{align}
Since $ e^{-\frac{x_n^2}s} \approx 1$ for $ x_n^2 < s<1 $,   it follows from    \eqref{phi-assume-v16} that
\begin{align}\label{eq1016-6}
\begin{split}
\int^\frac14_{x_n^2} s^{-\alpha}  e^{-\de_1\frac{x_n^2}s}\phi'(s) ds \approx  \int^\frac14_{x_n^2} s^{-\alpha}  \phi'(s) ds
\approx   x_n^{-2\alpha +2} \phi'(x_n^2).
   \end{split}
   \end{align}
Recalling that  ${\rm sgn} (\phi')  \phi' (s)> 0$ for $ 0 < s <\frac12$ and  $e^{-b} \leq c_m b^{-m}$ for any $ b > 0$ and $m \in {\mathbb N}$, we have
\begin{align}
\label{eq1016-7-1-1}
\begin{split}
  {\rm sgn} (\phi')  \int_{t-1}^{x_n^2} s^{-\alpha} e^{- \de_1\frac{x_n^2}s} \phi'(s) ds
&  \leq  c \, {\rm sgn} (\phi')  \int_{t-1}^{x_n^2} s^{-\alpha}  (\frac{x_n^2}{\de_1s})^{-m} \phi'(s) ds\\
&  \leq  c \, {\rm sgn} (\phi')  x_n^{-2m} \int_{t-1}^{x_n^2} s^{-\alpha + m } \phi'(s) ds\\
 &  \leq  c x_n^{-2m}  x_n^{-2\alpha + 2m +2}  {\rm sgn} (\phi')  \phi'(x_n^2)\\
&   \leq  c   x_n^{-2\alpha   +2} {\rm sgn} (\phi') \phi'(x_n^2).
 \end{split}
 \end{align}
Therefore,
we deduce from estimates \eqref{eq1016-5-1} -\eqref{eq1016-7-1-1} that
\begin{align}\label{eq1016-8-1}
 \int_{t-1}^\frac12 s^{ -\alpha } e^{-\de_1 \frac{x_n^2}s}  \phi'( s)  ds
 & \approx   x_n^{-2\alpha +2} \phi'(x_n^2).
\end{align}
Adding up estimates \eqref{240105-3}, \eqref{240105-1} , \eqref{240105-2} and \eqref{eq1016-8-1}, it follows in case $ t -1 < x_n^2$ that
\begin{align}\label{eq0103-10}
{\tilde\Psi}_{\phi,1}(t;2-t) \approx    x_n^{-2\alpha   +2} \phi'(x_n^2),
\end{align}
which implies \eqref{tildephi-10-1} for $ t -1 < x_n^2$.
\\
\\
$\bullet$ (The case $ \frac34 < t < 1$)\,\,
Noting that  $ \frac12 < 2t -1 < t$. Hence, we have
\begin{align}\label{240105-10}
{\tilde\Psi}_{\phi,1}(t)
&=  \int_0^\frac12   (t-s)^{-\alpha} e^{-\frac{x_n^2}{4(t-s)}}  \phi' ( 1-s)ds  + \int_\frac12 ^t (t-s)^{-\alpha} e^{-\frac{x_n^2}{4(t-s)}}  \phi' ( 1-s)ds.
  \end{align}
Since $\phi$ is smooth in $[\frac12,1]$, it is direct for $0 < x_n < \frac12$ that
\begin{align}\label{240105-11}
|  \int_0^\frac12   (t-s)^{-\alpha} e^{-\frac{x_n^2}{4(t-s)}}  \phi' ( 1-s)ds | \leq B_0 < \infty.
\end{align}
Using $\eqref{eq1113-1}_{1,2}$, it follows that 
\begin{align}
\begin{split}
\label{eq1017-10-1}
 &\int_\frac12 ^t (t-s)^{-\alpha} e^{-\frac{x_n^2}{4(t-s)}}  \phi' ( 1-s)ds\\
\approx &\int_\frac12^{2t -1}  (1-s)^{-\alpha} e^{-\de_1 \frac{x_n^2}{1-s}}  \phi' ( 1-s)ds  + \int_{2t-1}^t (t-s)^{-\alpha} e^{-\de_1\frac{x_n^2}{t-s}}  \phi' ( 1-t)ds\\
:=&I_1+I_2.
\end{split}
  \end{align}
In case that $ x_n^2 \le  1-t$,  direct computations via the change of variables give
\begin{align}\label{eq1016-9-1}
\begin{split}
I_2  &  =  \phi'  (1 -t) x_n^{2 -2\alpha}\int^\infty_{\frac{x_n^2}{1-t}}  s^{-2 + \alpha}  e^{- \de_1 s}  ds  \approx  \left\{\begin{array}{cl} \vspace{2mm}
\phi'( 1-t)  x_n^{2 -2\alpha}    \quad &\mbox{if}\,\,\,\al > 1,\\
 \phi'( 1-t)  ( 1-t)^{1 -\al}   \quad  &\mbox{if}\,\,\,\al < 1.
 \end{array}
 \right.
  \end{split}
  \end{align}
 Since $ e^{-\de_1 \frac{x_n^2}s} \approx 1$ for $ 1-t \leq s$, it follows from \eqref{phi-assume-v15} that
\begin{align}\label{eq1016-11}
\begin{split}
I_1
&\approx   \int_{2-2t}^1 s^{-\alpha}  \phi'(s) ds
    \approx
(1 -t)^{-\alpha +1} \phi'(1-t).
   \end{split}
\end{align}
From \eqref{240105-10} \eqref{240105-11},  \eqref{eq1017-10-1}, \eqref{eq1016-9-1} and    \eqref{eq1016-11}, we have
\begin{align*}
{\tilde\Psi}_{\phi,1}(t)
 & \approx   \left\{\begin{array}{cl} \vspace{2mm}
 \phi( 1-t)  x_n^{2 -2\alpha}    \quad &\mbox{if}\,\,\,\al > 1,\\
 (1 -t)^{-\alpha +1} \phi(1-t)    \quad  &\mbox{if}\,\,\,\al < 1.
 \end{array}
 \right.
\end{align*}
We obtain \eqref{tildephi-25-2} for $x_n^2 < 1-t$.

On the other hand, in case that $ 1-t < x_n^2$, we split $I_1$ as follows:
\begin{align}\label{eq1016-12}
I_1
  &  =   \int_{2-2t}^{x_n^2} s^{-\alpha} e^{-\frac{x_n^2}{4s}} \phi'(s) ds +  \int_{x_n^2}^\frac12 s^{-\alpha} e^{-\frac{x_n^2}{4s}} \phi'(s) ds:=I_{11}+I_{12}.
\end{align}
It is straightforwad from \eqref{eq1016-6} that
\begin{align}\label{eq1018-1}
\begin{split}
I_{12}
\approx  \int^\frac12_{x_n^2} s^{-\alpha}  \phi'(s) ds
 \approx
x_n^{-2\alpha +2} \phi'(x_n^2).
   \end{split}
   \end{align}
 With the same estimate of \eqref{eq1016-7-1-1}, we have
   \begin{align}
\label{eq1018-2}
\begin{split}
{\rm sgn} (\phi')  \int_{2-2t}^{x_n^2} s^{-\alpha} e^{-\frac{x_n^2}{4s}} \phi'(s) ds
 &  \leq  c  \, {\rm sgn} (\phi')  x_n^{-2\alpha  +2} \phi'(x_n^2).
 \end{split}
 \end{align}
Adding up \eqref{eq1016-12} and \eqref{eq1018-2}, we obtain   for   $ 1-t < x_n^2$
\begin{align}\label{eq0117-70-1}
I_1 \approx \int_\frac12^{2t -1}  (1-s)^{-\alpha} e^{-\de_1 \frac{x_n^2}{1-s}}  \phi' ( 1-s)ds
\approx   x_n^{-2\alpha +2} \phi'(x_n^2).
\end{align}
Direct computations via the change of variables give
\begin{align}\label{eq1016-9-1-1}
\begin{split}
I_2  &  \approx  \phi'  (1 -t) x_n^{2 -2\alpha}\int^\infty_{\frac{x_n^2}{1-t}}  s^{-2 + \alpha}  e^{-\de_1 s}  ds \approx \phi'( 1-t) (1-t)^{2-\al} x_n^{-2}
 e^{-\de \frac{x_n^2}{1-t}}.
  \end{split}
  \end{align}
Hence, from  \eqref{240105-10}, \eqref{240105-11}, \eqref{eq1017-10-1},  \eqref{eq0117-70-1}, \eqref{eq1016-9-1-1} and Lemma \ref{lemma240109}, we obtain \eqref{tildephi-25-2} for $ 1 -t \leq  x_n^2 $.
Following similar way of proving \eqref{phi-5} and \eqref{phi-50}, one can also obtain that 
\begin{equation}\label{March02-100}
\Psi_{\phi,1}(t)
\approx \abs{t-1}^{ 1-\al} \phi'(\abs{t-1}),\qquad \frac34<t<\frac54.
\end{equation}
which implies \eqref{March02-100-10} and \eqref{March02-100-20}.
Since its proof is tedious repetition, the details are omitted.
\end{proof}

Next, we provide the proofs of Lemma \ref{lemma240109} and Lemma \ref{lemma0709-1}.
\\
\begin{pflem22}
Let $k \in {\mathbb N} \cup \{0\}$ be an integer with $ 2^k \leq \frac{x_n^2}{t} < 2^{k+1}$.
Due to the Assumption \ref{Assume-phi}, it follows that
\begin{equation}\label{pf-Lemma22-10}
c^{-k-1}\phi(t)\le\phi(2^{k+1}t)\le \phi(x_n^2).
\end{equation}
We choose a positve integer $m$ such that $c\le 2^m$, where $c$ is the constant in \eqref{pf-Lemma22-10}.
Noting that $e^{-x}\le c_m x^{-m}$, we obtain
\[
e^{-\de_1 \frac{x_n^2}{t}}\phi(t) \le e^{-\de_1 2^k}c^{k+1}\phi(x_n^2) \le c_m (\delta_1 2^k)^{-m}c^{k+1}\phi(x_n^2) \le c_m c\delta_1^{-m}(\frac{c}{2^{m}})^k\phi(x_n^2)
\le c_m c\delta_1^{-m}\phi(x_n^2).
\]
Since $cc_m$ is an absolute constant, the first inequality in \eqref{0110-1} is proved.
For the second inquality, it follows from the following observation:
\[
c^{-k-1}\,{\rm sgn}(\phi'(t))\,\phi'(t)\le\,{\rm sgn}(\phi'(2^{k+1}t))\phi'(2^{k+1}t)\le {\rm sgn}(\phi'(x_n^2))\phi(x_n^2).
\]
Similar computations as above deduce the second inequality, and thus we skip its details.
In the case that $\phi = t^{-a} $, it is starightforward that $ t^{-1} \phi(t) \approx {\rm sgn} (\phi')  \phi'(t)$ for $ 0 < t< \frac12$. Again, the similar arguments as \eqref{0110-1} implies \eqref{240110-5}, and the details are omitted.
\end{pflem22}
\\
\begin{pflem24}
The second equality of \eqref{0515-1} was shown in  Lemma 3.2 in \cite{CK23}.
Thus it suffices to prove the first equality of \eqref{0515-1}.  Let $x'\in\Rn$ with $\abs{x'}\ge 1$. We divide $\Rn$ by three disjoint sets
$D_1, D_2$ and $D_3$, which are defined by
\[
D_1=\bket{z'\in\Rn: |x'-z'| \leq \frac1{10} |x'|},
\]
\[
D_2=\bket{z'\in\Rn: |z'| \leq \frac1{10} |x'|}, \qquad D_3=\Rn\setminus
(D_1\cup D_2).
\]
We then split the following integral into three terms as follows:
\begin{equation}\label{0730-2}
\int_{\Rn}  \Ga'(x'-z',t)   D_{x_n}  N( z', x_n) dz' =
\int_{D_1}\cdots + \int_{D_2} \cdots+ \int_{D_3}\cdots := K_1 +  I
+K_3.
\end{equation}
We first estimtate $I$.
\begin{align}\label{estJ2}
\begin{split}
  I & = \int_{|z'| \leq \frac1{10} |x'|} D_{x_n}  N( z',x_n)    \Ga'(x'-z',t)  \big)    dz' \\
&\  \approx    ct^{-\frac{n -1 }2  } e^{-\delta\frac{|x'|^2}{t}}     \int_{|z'| \leq \frac1{10} |x'|}  \frac{x_n  }{  (|z'|^2 + x_n^2)^{\frac{n}2} }  dz'\\
 &  \approx    c t^{-\frac{n-1 }2  } e^{-\delta\frac{|x'|^2}{t}}     \int_{|z'| \leq \frac{|x'|}{10x_n}}  \frac{1}{  (|z'|^2 + 1)^{\frac{n}2} }  dz'\\
& \approx   c      t^{-\frac{n-1 }2  }   e^{-\delta\frac{|x'|^2}{t}}.
\end{split}
\end{align}

On the other hand, the term $K_3$ is controlled as follows:
\begin{align}\label{est-J3}
\begin{split}
   |K_3| &   \leq   c\frac{x_n }{(|x'|^2 + x_n^2)^{\frac{n}2}  }  t^{-\frac{n-1}2  }  \int_{\{|z'| \geq \frac1{10} |x'|\}}  e^{-\frac{|z'|^2}{t}} dz'\\
  &   \leq  c \frac{x_n }{(|x'|^2 + x_n^2)^{\frac{n}2}  }   \int_{\{|z'| \geq \frac{ |x'|}{10 \sqrt{t}}\}}  e^{-|z'|^2} dz'\\
 & \leq c \frac{x_n }{(|x'|^2 + x_n^2)^{\frac{n}2}  }  e^{-\frac{|x'|^2}t}\le cx_n t^{\frac12},
 \end{split}
\end{align}
where we used that $e^{-\frac{|x'|^2}{2t}}\lesssim \frac{t^{\frac{m}2}}{\abs{x'}^m}$ for $ m \geq 0$ (in fact, we take $m=\frac{1}{2}$).

Meanwhile, we decompose $K_{1}$ in the following way.
 \begin{align*}
\notag  K_{1} & =     \int_{ D_1}    \Ga'(x' -z',t)  D_{x_n }  N(z',x_n) dz'
\\
\notag &
 =    \int_{\{|z'| \leq    \frac1{10}\frac{ |x'|}{\sqrt{t}}\}}     \Ga'(z',1) \Big( D_{x_n}  N(x' -\sqrt{t}z',x_n) - D_{x_n}  N(x',x_n) \Big)dz'\\
\notag &\quad + D_{x_n}   N(x',x_n) \int_{\Rn}  \Ga'(z',1)  dz' - D_{x_n}  N(x',x_n) \int_{\{|z'| \geq  \frac1{10}\frac{ |x'|}{\sqrt{t}}\}}  \Ga'(z',1)  dz'\\
& = K_{11} +D_{x_n}   N(x',x_n) + K_{12},
\end{align*}
where we used $\int_{\Rn}  \Ga'(z',1)  dz'=1$.
Since $\frac{|x'|^2}{t}\geq 1$, we observe that
\begin{equation}\label{0730-1}
\begin{split}
|K_{11} (x',t)| \leq & c\frac{x_n }{(|x'|^2 + x_n^2)^{\frac{n+1}2}  } t^\frac12   \int_{\{|z'| \leq  \frac1{10}\frac{ |x'|}{\sqrt{t}}\}}    e^{-|z'|^2}|z'|   dz'\\
\le & c\frac{x_n }{(|x'|^2 + x_n^2)^{\frac{n+1}2}  }t^{\frac{1}{2}  }\le cx_nt^{\frac{1}{2}}
\end{split}
\end{equation}
and
\begin{equation}\label{est-J122}
\begin{split}
|K_{12} (x',t)| \leq & c\frac{x_n }{(|x'|^2 + x_n^2)^{\frac{n}2}  }   \int_{\{|z'| \geq  \frac1{10}\frac{ |x'|}{\sqrt{t}}\}}    e^{-|z'|^2}   dz'\\
\le &c \frac{x_n }{(|x'|^2 + x_n^2)^{\frac{n}2}  } e^{-\frac{|x'|^2}{2t}}\le cx_nt^\frac12.
\end{split}
\end{equation}
Adding up \eqref{estJ2} -\eqref{est-J122} and putting $J_1:=K_3+K_{11}+K_{12}$, we deduce   \eqref{0515-1}-$\eqref{Jkl}$.
\end{pflem24}


\section{Proof of Theorem \ref{theo1114-1} }
\label{pf-SS}
\setcounter{equation}{0}

\subsection{Estimate of $w$}
We remind the decomposition of  $w$, which given as
$w = w^L + w^N + w^B$ in \eqref{1220-8}.  Let $ x' \in B'_1\subset\mathbb{R}^{n-1}$.
Noting that  $|x' -y'| \geq \frac12 $ for   $y' \in A$ and recalling boundary conditions in \eqref{0502-6}-\eqref{boundarydata} with Assumption \ref{Assume-phi}, it follows for $\frac34 < t <\frac54$ that
\begin{align}\label{0930-1}
\begin{split}
w^L_i(x,t) &  = \int_0^t \int_{A} L_{ni}(x'-y', x_n, t-s) g_n(y',s) dy' ds\\
&  \leq  c \int_0^t    (t -s)^{-\frac12} \phi ( 1-s) ds\\
&\approx  \max(1, \abs{t-1}^{\frac12} \phi (|t-1|) ), 
\end{split}
\end{align}
where we used \eqref{est-L-tensor}.

On the other hand, we note that $w^N_i$ vanishes for $t>1$. Furthermore, it follows from  \eqref{1220-8} that
\begin{align}\label{eq1017-25-1}
w^N_i (x,t) = 2\int_{A} D_{x_i} N (x'- y', x_n) g_n(y', t) dy' = 2 T_i  g^{\mathcal S}_n (x') \phi(1-t),
\end{align}
where
\[
T_i  g^{\mathcal S}_n (x')  = \int_{A}D_{x_i} N(x' -y', x_n)  g^{\mathcal S}_n (y')dy'.
\]
Since $T_i  g^{\mathcal S}_n (x')$ is positive and finite, we have $T_i  g^{\mathcal S}_n (x')  \approx 1$ for $|x'|  <  \frac12$ and $ x_n \in(0,  \frac12)$, which implies via \eqref{eq1017-25-1} that
\begin{align}\label{eq1017-25-1-1}
w^N_i (x,t) \approx \phi(1-t){\bf I}_{\bket{t<1}}.
\end{align}

It remains to estimate $w_i^{\mathcal B}$.
Since $w_i^{\mathcal B}=0$ for $i=n$,
we note that $|x' -y'| \geq 1$ for $|x'| <\frac12 $ and $y' \in A$. Using  \eqref{eq1214-1} and \eqref{eq1214-1-1},  for $i\neq n$, we have for $\frac34< t< \frac54$
\begin{align}\label{eq1017-22}
\begin{split}
 w_i^{\mathcal B}(x,t)
   & =   \int_0^t   D_{x_n} \Ga_1(x_n, t-s)  g_n^{\mathcal T} (s) \int_{A}  \int_{\Rn}    \Ga'(x'-y'-z',t)    D_{z_i }  N( z',0) dz'  g^{\mathcal S}_n(y') dy'ds\\
    & \approx   \int_0^t   D_{x_n} \Ga_1(x_n, t-s)  g_n^{\mathcal T} (s)  ds.
  \end{split}
\end{align}
We separately treat \eqref{eq1017-22} for the cases that $t<1$ and for $t>1$.
\\
\\
$\bullet$ {\bf (The case  $ 1 -t_0 < t <1$)}\,\,
Using  \eqref{eq1017-22}  and \eqref{tildephi-25},  we observe that
\begin{align}\label{eq1017-22-2}
\begin{split}
  w_i^{\mathcal B}(x,t)
   & \approx  -\int_0^t  \frac{x_n}{(t -s)^\frac32} e^{-\frac{x_n^2}{4(t-s)}}    \phi(1 -s)  ds\\
   & \approx  -\left\{\begin{array}{ll} \vspace{2mm}
    \phi (1 -t)  \quad &\mbox{if}\,\,x_n  < ( 1-t)^\frac12,\\
     \phi(x_n^2)+   \phi( 1-t) (1-t)^\frac12 x_n^{-1}   e^{-\de \frac{x_n^2}{1-t}}
     \quad  &\mbox{if}\,\,(1 -t)^\frac12 < x_n.
    \end{array}
    \right.
  \end{split}
\end{align}
Therefore, we have 
\begin{align}\label{eq1017-22-3}
\begin{split}
  \abs{w_i^{\mathcal B}(x,t) }
  \lesssim \left\{\begin{array}{ll} \vspace{2mm}
     1 \quad  &\mbox{if}\,\,\phi(t)=\abs{\ln t}^{-1}\\
    \phi (1 -t)
    \quad &\mbox{otherwise}.
    \end{array}
    \right.
  \end{split}
\end{align}
Adding up \eqref{0930-1}, \eqref{eq1017-25-1-1} and \eqref{eq1017-22-3}, we obtain
\begin{align}\label{eq1017-22-4}
\begin{split}
  \abs{w_i(x,t) }
  \lesssim \left\{\begin{array}{ll} \vspace{2mm}
     1 \quad  &\mbox{if}\,\,\phi(t)=\abs{\ln t}^{-1}\\
    \phi (1 -t)
    \quad &\mbox{otherwise}.
    \end{array}
    \right.
  \end{split}
\end{align}
We complete the proof of \eqref{eq1218-10} for $ t < 1$.


Let $\phi\neq \abs{\ln t}^{-1}$ and we choose a sufficiently large $\ep_1 \gg 1$. We then show that there exists a sufficiently small $\epsilon$ depending on $\epsilon_1$  such that $w^B(x,t) >- \ep \phi(1-t)$  in $\tilde{\calD}_-=\bket{(x,t)\in \calD_-:  \ep_1 \sqrt{1-t} < x_n<\frac12 }$.
Indeed, recalling the second estimate in \eqref{eq1017-22-2}, it follows that
\begin{equation}\label{Feb28-10}
w_i^{\mathcal B}(x,t) \gtrapprox -\phi(\epsilon_1^2 (1-t))-   \phi( 1-t)\epsilon_1^{-1}   e^{-\de \epsilon_1^2} \ge -\epsilon \phi( 1-t).
\end{equation}
In a similar way, we can see via \eqref{0930-1} that
\begin{equation}\label{Feb28-20}
|w^L(x,t)| < \ep \phi(1-t), \qquad \mbox{ for }\,\,(x,t)\in \tilde{\calD}_-.
\end{equation}
Summing up estimates \eqref{eq1017-25-1-1}, \eqref{Feb28-10} and \eqref{Feb28-20}, we conclude in  $\tilde{\calD}_-$ that
\begin{align}\label{1215-10-20}
w_i(x,t)\ge w^N_i (x,t) + w^B_i (x,t) -\abs{w^L_i (x,t) }\ge    ( c- 2 \ep ) \phi(1-t) \ge \frac{c}{2} \phi(1-t).
\end{align}
We complete the proof of \eqref{eq0828-5-1} for $ t<1$.
\\
\\
$\bullet$ {\bf (The case $ 1 < t < 1 + t_0 $)}\,\,
With the aid of \eqref{tildephi-15} and \eqref{eq1017-22}, we have
\begin{align}\label{eq1017-29}
\begin{split}
  w_i^{\mathcal B}(x,t)
   & \approx  -\int_0^1 \frac{x_n}{(t -s)^\frac32} e^{-\frac{x_n^2}{4(t-s)}}    \phi(1 -s)  ds\\
   & \approx  -\left\{\begin{array}{ll} \vspace{2mm}
    x_n (t-1)^{-\frac12} \phi(t-1)  \quad  &\mbox{if}\,\,x_n^2 \leq t-1,\\
    \max(1, \phi(x_n^2) ) \quad  &\mbox{if}\,\, t-1 \leq x_n^2.
    \end{array}
    \right.
  \end{split}
\end{align}
Owing to \eqref{0930-1},  \eqref{eq1017-25-1-1} and \eqref{eq1017-29},   it follows that
\begin{align}\label{eq1017-22-20}
\begin{split}
  \abs{w_i(x,t) }
  \lesssim \left\{\begin{array}{ll} \vspace{2mm}
     1 \quad  &\mbox{if}\,\,\phi(t)=\abs{\ln t}^{-1}\\
    \phi (1 -t)
    \quad &\mbox{otherwise},
    \end{array}
    \right.
  \end{split}
\end{align}
which implies \eqref{eq1218-10} for $t \in ( 1 , \frac54)$ and thus together with \eqref{eq1017-22-4}, the proof of \eqref{eq1218-10} is completed.
Moreover, if $\phi\neq \abs{\ln t}^{-1}$, by choosing sufficiently small   $t_0$ such that for $ \ep_1\sqrt{t-1}< x_n $ and $ \sqrt{t-1} < t_0$, we have
\begin{align}\label{1215-10-10}
w^B_i (x,t) + w^L_i (x,t) \le   \big(- c+ c_1 \sqrt{t-1}\big)  \phi(t-1) \le -\frac{c}{2}\phi(t-1).
\end{align}
On the other hand, from \eqref{0930-1} and  \eqref{eq1017-29}, for $ 1 < t < 1 + t_0$ and  $\frac12 \sqrt{t -1} < x_n < \sqrt{t -1}$, we have
\begin{align*}
w_i (x,t) \leq (-\frac{c}{2} + \ep) \phi(t-1) \leq -\frac{c}{4} \phi(t-1).
\end{align*}
This completes the proof of \eqref{eq0828-5-1-1}.
\qed



\subsection{Estimate of $\na w$}
We note for $ 1 \leq i \leq n-1$ that
\begin{align}\label{0706-1}
  D_{x_n} w^L_i (x,t)
& = - \sum_{k =1}^{n-1} D_{x_k}  w^L_{ik} (x,t)      + \frac12   D_{x_i}    f(x,t),
\end{align}
where 
\begin{align*}
f(x,t) &= \int_0^t \int_{\Rn} D_{x_n} \Ga(x' -z', x_n, t-s) g_n(z',s) dz'ds,\\
w^L_{ik} (x,t)  & = \int_0^t \int_{\Rn} L_{ik} (x' -y', x_n, t-s) g_n(y',s) dy' ds,\quad k=1, 2, \cdots, n-1.
\end{align*}
Similar estimates as in  \eqref{0930-1} imply for $\frac34 < t < \frac54$ that
\begin{align}\label{eq1129-1}
D_{x_n} w^L_i (x,t) &  \leq  c
\max(1,  |t -1|^{\frac12} \phi (|t-1|) ).
\end{align}
It is also straightforward as in \eqref{eq1017-25-1-1} that
 \begin{align}\label{eq1129-2}
 |D_x w^N_i (x,t)| \leq c \phi(\abs{t-1}){\bf I}_{\bket{t<1}}.
 \end{align}
We can easily see that tangential derivatives of $w^{\mathcal B}_i$ satisfy essentially the same estimates as those of $w^{\mathcal B}_i$, namely $| D_{x'} w_i^{\mathcal B}(x,t)|$ satisfies  \eqref{eq1017-22-3}.
Therefore, it remains to estimate $D_{x_n } w_i^{\mathcal B}$. As before, we treat the cases that $t<1$ and $t>1$ seperately.
\\
\\
$\bullet$ {\bf (The case $ 1- t_0< t< 1$)}\,\,
We recall that
\begin{align*}
B_{in}(x', x_n, t) = D_{x_n} A_i (x', x_n, t),
\end{align*}
where
\begin{align}\label{240103-1}
A_i(x,t) = \int_{\Rn} \Ga (x'-z', x_n, t) D_{z_i} N(z',0) dz',\qquad i=1,2,\cdots, n-1.
\end{align}
Using  that $(D_t -\De) A_i =0$ for $ (x,t) \in \R_+ \times (0, \infty)$ and  $g_n^{\mathcal T} (y',0) =0$, we have
\begin{align}\label{1216-1}
\begin{split}
D_{x_n} w_i^{\mathcal B}(x,t) & =  \int_0^t \int_{\Rn}    D_{x_n} B_{in}(x' -y', x_n, t-s)  g_n^{\mathcal T} (y',s) dy' ds\\
& =  \int_0^t \int_{\Rn}  (   D_{t}  - \De'_x) A_{i}(x' -y', x_n, t-s)  g_n^{\mathcal T} (y',s) dy' ds\\
& =  \int_0^t \int_{\Rn}  A_{i}(x' -y', x_n, t-s)  (   D_s  - \De'_y) g_n^{\mathcal T} (y',s) dy' ds.
\end{split}
\end{align}
Due to the estimates \eqref{eq1214-1},  and \eqref{tildephi-25},
the second term is estimated as follows:
\begin{align}
\begin{split}
&\abs{\int_0^t \int_{\Rn}  A_{i}(x' -y', x_n, t-s)  \De'_yg_n^{\mathcal T} (y',s) dy' ds}\\
\lesssim&\int_0^t \frac1{(t-s) ^\frac12} e^{-\frac{x_n^2}{4(t-s)}}  \phi (1 -s) ds  \\
\approx&\left\{\begin{array}{cl}
 \max(1,  x_n \phi(x_n^2)) + \phi( 1-t)  (1-t)^\frac32 x_n^{-2}  e^{-\de_1 \frac{x_n^2}{1-t}}  \quad &\mbox{if}\,\, 1-t< x_n^2,    \vspace{2mm} \\
 \max(1, (1 -t)^{\frac12} \phi(1-t))  \quad &\mbox{if}\,\, x_n^2 < 1-t.
 \end{array}
 \right.
 \end{split}
\end{align}
On the other hand, due to \eqref{eq1214-1},  \eqref{eq1214-1-1}  and \eqref{tildephi-25-2}, we compute for the first term as follows:
\begin{align}
\begin{split}
& \int_0^t \int_{\Rn}  A_{i}(x' -y', x_n, t-s)  D_s  g_n^{\mathcal T} (y',s) dy' ds\\
\approx& -\int_0^t \frac1{(t-s) ^\frac12} e^{-\frac{x_n^2}{4(t-s)}}   \phi'(1 -s) ds \\
\approx & - \left\{\begin{array}{cl}\vspace{2mm}
 x_n  \phi' (x_n^2)  + \phi'( 1-t) (1-t)^\frac32 x_n^{-2}  e^{-\de_1 \frac{x_n^2}{1-t}}\quad &\mbox{if}\,\, 1-t< x_n^2,\\
 (1 -t)^{\frac12} \phi'(1-t)  \quad &\mbox{if}\,\,x_n^2 < 1-t.
 \end{array}
 \right.
 \end{split}
\end{align}
Thus,  using Lemma \ref{lemma240109},  it follows that
\begin{align}\label{eq1129-3}
|D_{x_n} w^B_i(x,t)|
&\leq -c   ( 1-t)^\frac12 \phi'(1-t) {\bf I}_{\calD_+} + x_n \phi'(x_n^2)  {\bf I}_{\calD_-}.
 \end{align}
Hence, we obtain the second part in \eqref{eq0829-1-8}.
In addition, we also note that
\begin{align}
\begin{split}
D_{x_n} w_i(x,t)
 & \approx  \left\{\begin{array}{cl}\vspace{2mm}
  - ( 1-t)^\frac12 \phi'(1-t)   \quad &\mbox{ if }\,\, 0<1-t<\min\bket{t_0, t_1} \mbox{ and } (x,t)\in \tilde{ {\calD}_+}\,\\
 -x_n \phi'(x_n^2)  \quad&\mbox{ if }\,\, 0<1-t<\min\bket{t_0, t_1} \mbox{ and } (x,t)\in \tilde { {\calD}_-}.
 \end{array}
 \right.
  \end{split}
 \end{align}
This completes the proof of \eqref{1218-20}.
\\
$\bullet$ {\bf (The case $ 1 < t <  1 + t_0$)}
\\
We first recall that
\begin{align}\label{eq1018-55}
  \begin{split}
D_{x_n} w_{i}^{\mathcal B}(x,t) & =     \int_0^1  \int_{\Rn} D^2_{x_n} A_i (x' -y', x_n, t-s)  g_n (y',s)  dy' ds,
  \end{split}
\end{align}
where $A_i$ is defined in \eqref{240103-1}. Since $|D^2_{x_n} A_i (x' -y', x_n, t-s)| \leq c  (t -s)^{-\frac32} e^{-\frac{x_n^2}{4(t-s)}}$,  it follows from \eqref{tildephi-15} that
\begin{align*}
|D_{x_n} w_{i}^{\mathcal B}(x,t) | \leq c \int_0^1  (t -s)^{-\frac32} e^{-\frac{x_n^2}{4(t-s)}} \phi(1 -s) ds \leq c \bke{(t-1) \vee x_n^2}^{ -\frac12} \phi\bke{(t-1) \vee x_n^2}).
\end{align*}
 We complete the proof of the first part in \eqref{eq0829-1-8}.

In case that $ x_n \leq \ep_0 \sqrt{ t -1}$, it is direct that   for $0<s<1$,
\[
D^2_{x_n} \Ga_1(x_n, t-s) \approx -(t -s)^{-\frac32} e^{-\de_1 \frac{x_n^2}{t-s}} \approx -(t -s)^{-\frac32}. 
\]
Hence, it follows from \eqref{eq1214-1}, \eqref{eq1214-1-1} and \eqref{phi-5}  that
\begin{align} \label{eq1018-30-1}
\begin{split}
 D_{x_n} w_{i}^{\mathcal B}(x,t)
 \approx    - (t -1)^{-\frac12} \phi (t -1).
   \end{split}
\end{align}
Combining \eqref{eq1129-1} and \eqref{eq1018-30-1},  for $ x_n \leq \ep_0 \sqrt{ t -1}$, we have
\begin{align*} 
\begin{split}
 D_{x_n} w_{i} (x,t)
 \approx    - (t -1)^{-\frac12} \phi (t -1),
   \end{split}
\end{align*}
which completes the proof of \eqref{0126-1}  for $ x_n \leq \ep_0 \sqrt{ t -1}$.

Next, we consider the case that $ \ep_1 \sqrt{t-1} < x_n  $.
Dividing time interval into two parts such as $0<s<2-t$ and $2-t<s<1$, we split the integral
 \begin{align}\label{eq1018-55}
  \begin{split}
D_{x_n} w_{i}^{\mathcal B}(x,t) =&     \int_0^{ 2-t}  \int_{\Rn} D^2_{x_n} A_i (x' -y', x_n, t-s)  g_n (y',s)  dy' ds \\
 &+    \int_{ 2-t}^1 \int_{\Rn} D^2_{x_n} A_i (x' -y', x_n, t-s)  g_n (y',s)  dy'  ds\\
 :=&J_1+J_2.
  \end{split}
\end{align}
Noting that $  D^2_{x_n} \Ga_1(x_n, t-s)  \approx  (t -1)^{-\frac52}x_n^2 e^{-\frac{x_n^2}{4(t-1)}} $ for    $  2 -t < s$,    we obtain
  \begin{align}\label{240103-2}
  \begin{split}
 J_2
&\approx   (t -1)^{-\frac52}x_n^2 e^{-\frac{x_n^2}{4(t-1)}} \int_{ 2-t}^1  \phi(1 -s) ds \approx      (t -1)^{-\frac32}x_n^2 e^{-\frac{x_n^2}{4(t-1)}}    \phi(t-1 ).
  \end{split}
\end{align}
Using  that $(D_t -\De) A_i =0$ for $ (x,t) \in \R_+ \times (0, \infty)$.
Since $\phi(1)=0$ and integrating $J_1$ by parts, we get
 \begin{align}\label{eq1018-55-1}
  \begin{split}
J_1
=  &\int_0^{ 2-t}  \int_{\Rn}  (D_t -\De'_x) A_i (x' -y', x_n, t-s)  g_n (y',s)  dy' ds \\
   =&   \int_0^{ 2-t}  \int_{\Rn} A_i (x' -y', x_n, t-s)   D_s   g_n (y',s)  dy' ds  \\
   &   -  \int_0^{ 2-t}  \int_{\Rn} A_i (x' -y', x_n, t-s)   \De'_y  g_n (y',s)  dy' ds  \\
  & - \int_{\Rn} A_i (x' -y', 2t -2) \phi(t-1) g^{\mathcal S}_n(y') dy'\\
  :=&J_{11}+J_{12}+J_{13}.
  \end{split}
\end{align}
It is straightforward from \eqref{eq1214-1} and \eqref{eq1214-1-1} that $J_{13}$ satisfies
\begin{align}
J_{13}=- \int_{\Rn} A_i (x' -y', 2t -2) \phi(t-1) g_n(y') dy' & \approx -\frac1{(t -1)^\frac12} e^{-\frac{x_n^2}{ 4(t -1)}} \phi(t -1).
\end{align}
From the proof of \eqref{tildephi-15} , the second term $J_{12}$ is controlled as follows:
\begin{align}
\begin{split}
\abs{J_{12}}
&\le\abs{\int_0^{ 2-t}  \int_{\Rn} A_i (x' -y', x_n, t-s)  \De'_y  g_n (y',s)  dy' ds }\\
 &  \leq c \int_0^{2 -t} \frac{1}{(t -s)^\frac12} e^{-\frac{x_n^2}{4(t -s)}} \phi(1 -s) ds \\
 &  \leq c \int_0^{2 -t} \frac{1}{(1 -s)^\frac12} e^{-\de_1\frac{x_n^2}{1 -s}} \phi(1 -s) ds\\
 &  = c \int_{t-1}^1 \frac{1}{s^\frac12} e^{- \de_1\frac{x_n^2}{s}} \phi(s) ds\\
 & \leq c\max( 1,   x_n  \phi (  x_n^2) ).
 \end{split}
\end{align}
Due to \eqref{tildephi-10-1}, we have
\begin{align}\label{240103-3}
J_{11}
& \approx  -\int_0^{ 2-t}  \frac1{(t -s)^\frac12} e^{-\frac{x_n^2}{4( t -s)}}  \phi'(1-s) ds 
\approx -x_n \phi' ( x_n^2).
\end{align}
We note for $\ep_1 \sqrt{t-1}  < x_n$ with  sufficiently large $\ep_1 >0$ that
\begin{align*}
 (t -1)^{-\frac32}x_n^2 e^{-\frac{x_n^2}{4(t-1)}}    \phi(t-1 ) -  (t -1)^{-\frac12}  e^{-\frac{x_n^2}{4(t-1)}}    \phi(t-1 )  & \approx (t -1)^{-\frac32}x_n^2 e^{-\frac{x_n^2}{4(t-1)}}    \phi(t-1 ) ,\\
  \max( 1,   x_n  \phi (  x_n^2) ) - x_n \phi' (x_n^2)  & \approx - x_n \phi' (x_n^2).
 \end{align*}
From  \eqref{eq1018-55} to \eqref{240103-3} , 
for $ t -1 < t_0$ for small $ t_0$,  we have
 \begin{align}\label{eq240103-4}
 \begin{split}
D_{x_n} w_{i}^{\mathcal B}(x,t)    & \approx   - x_n \phi' (x_n^2) + (t -1)^{-\frac32}x_n^2 e^{-\frac{x_n^2}{4(t-1)}}    \phi(t-1 ).
 \end{split}
\end{align}
With the aid  of \eqref{eq1129-1},  \eqref{eq240103-4} and Lemma \ref{lemma240109}, we  complete the proofs of  \eqref{Jan02-10}  for $  \ep_1 \sqrt{ t -1} \leq x_n $.
\qed


\subsection{Estimate of pressure}

For convenience, we denote
 \begin{align*}
 {\mathcal U} g_n(x',0,t) = \int_0^t \int_{A} \Ga(x' -y', 0, t-s) g_n(y',s)dy'ds.
 \end{align*}
Recalling that $ g_n (x', 0) =0$ and \eqref{representationpressure}, we decompse $p = p_1 + p_2 + p_3$, where
\begin{align}\label{eq1123-1}
\begin{split}
p_1(x,t) & = -   D_{x_n} D_{x_n} N *'  g_n(x, t),\\
p_2(x,t) & =       2  (D_t -\De')  D_{x_n} N *' {\mathcal U} g_n(x,t) \\
&=  2  D_t D_{x_n} N *' {\mathcal U} g_n(x,t)
-  2\De'  D_{x_n} N *' {\mathcal U} g_n(x,t) \\
&:= p_{21} + p_{22},\\
p_3(x,t) & =   N *'  D_t g_n (x, t).
\end{split}
\end{align}
Since $|x' - y'| \geq 1$ for $|x'| < \frac12$ and $y' \in A$, we can estimate $p_1$ and $p_3$ as follows:
\begin{align}\label{p1018-1}
\begin{split}
|p_1(x,t)| &\le  c \phi (1 -t) {\bf I}_{\bket{0 < t < 1}},\\
p_3(x,t) &\approx - \phi'(1-t) {\bf I}_{\bket{0 < t < 1}}.
\end{split}
\end{align}
It remains to estimate $p_2$.  Firstly, for $p_{22}$ we note that
\begin{align}\label{p1018-2}
\begin{split}
p_{22}(x,t) & =-   2   \De'  D_{x_n} N *' {\mathcal U} g_n(x,t)=-   2     D_{x_n} N *' {\mathcal U} \De' g_n(x,t)\\
& = -2\int_0^t \int_{A} ( t -s)^{-\frac12} \De'  g_n(y',s) \int_{\Rn}  \Ga'(x' -y'-z', t-s) D_{x_n} N(z', x_n)dz' dy'ds.
\end{split}
\end{align}
From  \eqref{p1018-2} and $\eqref{240116-1} $, we have
\begin{align}\label{p1018-3}
\begin{split}
\abs{p_{22}(x,t)}
& \leq         c   \int_0^t  (t -s)^{-\frac12}  \big( x_n    + (t -s)^{-\frac{n-1}2} e^{-\frac1{(t-s)} }      \big)    g^{\mathcal T} _n(s)        ds \\
&\leq c \big( 1 + x_n \max(1, |t -1|^\frac12 \phi(|t -1|))  \big)\\
&\leq c \max(1, x_n|t -1|^\frac12 \phi(|t -1|).
\end{split}
\end{align}
$\bullet$ {\bf (The case  $\frac34 < t < 1$)}
\\
Since $g_n (y', 0) =0  $, by intergration by parts and \eqref{240116-1-8},  we have
\begin{align}\label{p1018-4}
\begin{split}
p_{21}(x,t)
&=\int_0^t \int_{A} ( t -s)^{-\frac12} \phi'(1-s)  \int_{\Rn}  \Ga'(x' -y'-z', t-s) D_{x_n} N(z', x_n)dz'    g_n(y')       dy'ds\\
 &\approx    \int_0^t   ( t -s)^{-\frac12} \phi' (1-s) \big( x_n + (t -s)^{-\frac{n-1}2} e^{- \de \frac1{(t-s)} }   \big) ds\\
  &\approx   \max(1,  x_n ( 1-t)^\frac12 \phi'(1-t)).
 \end{split}
 \end{align}
Therefore, adding up \eqref{p1018-1}, \eqref{p1018-3}  and \eqref{p1018-4},  we otain
for $  x_n < \frac12$ and $ 1 -t < t_1$,  
\begin{align}\label{pressure-100}
p(x,t) \approx -c    \phi'(1-t).
 \end{align}
We prove the first part of \eqref{1219-3} for $ t < 1$.
\\
\\
$\bullet$ {\bf (The case $ 1 < t$)}\,\,
Since $p_{1}(x,t) = p_3(x,t) =0$ for $t>1$,
 Lemma \ref{lemma0709-1} implies that
\begin{align}
\begin{split}
 p_{22}  & = -2\int_0^1 \int_{\Rn} ( t -s)^{-\frac12} \De'  g_n(y',s) \big(D_{x_n} N(x'-y', x_n) + (I+J_1)(x'-y', x_n, t)        \big) dy'ds\\
 & \leq c  \int_0^1  \big( x_n ( t -s)^{-\frac12}     + (t -s)^{-\frac{n}2} e^{-\frac1{8(t-s)} }      \big)   \phi( 1-s)  ds \\
 &  \leq c  \big( 1  + x_n\max(1, (t -1)^\frac12 \phi(t -1)) \big)
 \end{split}
\end{align}
and
\begin{align}
\begin{split}
  p_{21}& = \int_0^1 \int_{\Rn} ( t -s)^{ -\frac32}   \phi(1-s) \int_{\Rn}  \Ga'(x' -y'-z', t-s) D_{x_n} N(z', x_n)dz'    g_n(y')       dy'ds\\
& \approx     \int_0^1  \big( x_n ( t -s)^{ -\frac32}  +  (t -s)^{-\frac{n +1}2} e^{-\frac1{(t-s)} }     \big)\phi(1-s) ds\\
&  \approx  x_n( t-1)^{ -\frac12}\phi(t-1) +1.
\end{split}
\end{align}
Hence, we have
\begin{align}\label{0301-1-1}
c_1x_n  (t-1)^{-\frac12 } \phi(t-1)  -c_2  \leq p_{21} (x,t) \leq c_3 x_n  (t-1)^{-\frac12 } \phi(t-1)  + c_4,
\end{align}
which implies the second part of \eqref{1219-3} for $t>1$, and thus we complete the proof of \eqref{1219-3}.
 \qed

\section{Proof of Proposition \ref{theo1001}}
\label{NSequations}
\setcounter{equation}{0}

To prove Propositiion \ref{theo1001}, we consider the following equations
\begin{align}\label{1001-1}
\left\{\begin{array}{l}\vspace{2mm}
W_t -\De W + \na Q = f  \quad \R_+ \times (0, T),\\
\vspace{2mm}
{\rm div} \, W =0 \quad \R_+ \times (0,T),\\
W|_{t =0} =0, \quad W|_{x_n =0 } =0,
\end{array}
\right.
\end{align}
where $ 0 < T \leq \infty$.
We recall some known results for the system \eqref{1001-1}.

\begin{prop}\label{theo0503}(\cite[Proposition 2.3]{CK20})
Let $1< p, q< \infty$. Suppose that $f = {\rm div} \, {\mathcal F}$ with $F \in L^q (0,\infty; L^p (\R_+))$, $ {\mathcal F}|_{x_n =0} =0 $.
Then, the solution $W$ of \eqref{1001-1}  satisfies
\begin{align}\label{0504-1}
\|  \na W\|_{L^q (0, \infty;   L^p (\R_+))} \leq c  \| {\mathcal F}\|_{L^{q} (0,\infty;  L^p (\R_+))}.
\end{align}
\end{prop}

\begin{prop}\label{thm-stokes}(\cite[Theorem 1.2]{CJ})\,\,
Let  $1<p, q<\infty$. Assume that  $  f =  {\rm div} {\mathcal F}$ with  ${\mathcal F}\in
L^{q_1}(0,\infty; L^{p_1}(\hR)) $, where $p_1, q_1$ satisfy
$n/p_1+2/q_1=n/p+2/q+1$ with $q_1\leq q$ and $p_1 \leq p$. Then, there is the unique
very weak solution $W\in L^q(0,\infty;L^p(\R_+))$ of the Stokes equations
\eqref{1001-1}  such that $W$
satisfies
\[
\| W\|_{L^q(0,\infty;L^p(\R_+))}\leq c \|   {\mathcal F}\|_{L^{q_1}(0,\infty; L^{p_1}(\R_+))}.
\]
\end{prop}

\begin{prop}\label{theoexternel-boundary}\cite[Theorem 2.2]{CCK}
Let $ n +2 <  p < \infty $. Let $ f= {\rm div} \, {\mathcal F}$ with  $ {\mathcal F}
\in L^p (\R_+ \times (0, T)) $.   Then, there exists unique  very weak solution, $w \in L^\infty (\R_+ \times (0, T))$, of the Stokes equations
\eqref{1001-1} such that
\begin{equation*}
\| W\|_{L^\infty (\R_+ \times (0, T))} \leq cT^{\frac12(1-\frac{n+2}{p} )} \|  {\mathcal F}  \|_{ L^p (\R_+ \times (0, T))}.
\end{equation*}
\end{prop}

\begin{prop}\label{prop0503}\cite[Proposition 3.4]{CK3}
Suppose  that
$f\in L^{q_1} (0, 1; L^{p_1} (\R_+))$ with $1<p_1, q_1<\infty$.
If $p$ and $q$ satisfy
$p>\frac{n}{n-1}p_1 $ and $\frac2{q_1}+\frac{n}{p_1}-1 \leq  \frac2q  + \frac{n}{p}$,
then
\begin{equation}\label{pi-bound-2}
\| Q\|_{L^{q}(0, 1; L^p (\R_+))} < c \| f\|_{L^{q_1} (0, 1; L^{p_1} (\R_+))}.
\end{equation}
\end{prop}

Our aim is to establish the existence of solution  $v$ for
\eqref{CCK-Feb7-10} satisfying $v \in L^{2q} ((\R_+ \times (0, \infty))$
and $\na v \in L^q(\R_+ \times (0, \infty))$ for  fixed $ q >\frac{n+2}{2}$.   In order to do that, we
consider the iterative scheme for \eqref{CCK-Feb7-10}, which is
given as follows: For a positive integer $m\ge 1$
\begin{align*}
&v^{m+1}_t-\Delta v^{m+1}+\nabla q^{m+1}=-{\rm
div}\,\left(v^{m}\otimes v^{m}+v^{m}\otimes U+U\otimes
v^{m}+U\otimes U \right),\\
& \qquad \qquad \qquad \qquad \qquad {\rm div} \, v^{m+1} =0
\end{align*}
with homogeneous initial and boundary data, i.e. $v^{m+1}(x,0)=0$ and
$v^{m+1}(x,t)=0$ on $\{x_n=0\}$.

Let $ q > \frac{n+2}2$.
We set $v^1=0$. We then have, due to Proposition \ref{theo0503} and Proposition \ref{thm-stokes},
\begin{align}\label{est-v2-10}
\begin{split}
  \|\na  v^2\|_{L^q   (\R_+ \times (0, T ))}  &\leq c     \||U|^2\|_{L^q   (\R_+ \times (0, T )) }
  \leq c   \|U\|^2_{L^{2q}   (\R_+ \times (0, T ))}
 \leq c  \|w\|^2_{L^{2q} },\\
\| v^2\|_{L^{2q}( \R_+ \times (0, T))}&\leq c \||U|^2\|_{L^{q_1}(0,T; L^q ( \R_+ )) } \leq c  T^{\frac{1}{2}(1-\frac{n+2}{2q})}\||U|^2\|_{  L^q ( \R_+ \times (0, T)) }  \\
&\leq c    T^{\frac{1}{2}(1-\frac{n+2}{2q})}\|w\|^2_{L^{2q} },
\end{split}
\end{align}
where $\frac2{q_1} +\frac{n}q =\frac{n+2}{2q} +1$.
We have  $ A :=      \|w\|_{L^{2q} } + \|\na w\|_{L^{q} }    \leq c\al$, where $\al>0$ is defined in \eqref{0502-6}. We take $ T =1$.  Taking $\al >0$ small such that
$A<\frac{1}{8c}$, where $c$ is the constant in
\eqref{est-v2-10} such that
\begin{align*}
\|\na  v^2\|_{L^{q}( \R_+ \times (0, 1)) }  +   \| v^2\|_{L^{2q}( \R_+ \times (0, 1)) }  \le cA^2<A.
\end{align*}
Then, iterative arguments show that
\begin{align}\label{0531-1}
\begin{split}
 \| v^{m+1}\|_{L^{2q}((\R_+ \times (0, 1)) } &  \leq c
\||v^m|^2+|v^mU|+|U|^2\|_{L^{q}((\R_+ \times (0, 1)) }\\
&\leq  2c\left(\|v^m\|^2_{L^{2q}((\R_+ \times (0, 1)) }+\|w\|^2_{L^{2q}  }\right)   \leq 4c A^2 <A,\\
 \| \na v^{m+1}\|_{L^q (\R_+ \times (0, 1))}
  &  \leq c \big(
\||v^m|^2+|v^mU|+|U|^2\|_{L^q((\R_+ \times (0, 1)) }   \big) \\
 &\quad \leq  2c\left(\|v^m\|^2_{L^{2q}((\R_+ \times (0, 1)) }+\|w\|^2_{L^{2q}((\R_+ \times (0, 1))}  \right)\\
 & \quad   \leq 4c A^2 <A.
 \end{split}
\end{align}
We denote $V^{m+1}:=v^{m+1}-v^{m}$ and $Q^{m+1}:=q^{m+1}-q^{m}$ for $m\ge
1$. We then see that $(V^{m+1}, Q^{m+1})$ solves
\[
V^{m+1}_t-\Delta V^{m+1}+\nabla Q^{m+1}=-{\rm
div}\,\left(V^{m}\otimes v^{m}+v^{m-1}\otimes V^{m}+V^{m}\otimes
U+U\otimes V^{m}\right),
\]
\[
{\rm div} \, V^{m+1} =0
\]
with homogeneous initial and boundary data, i.e. $V^{m+1}(x,0)=0$
and $V^{m+1}(x,t)=0$ on $\{x_n=0\}$. It follows from \eqref{0531-1} that
\begin{align*}
\| \na V^{m+1}\|_{L^{q}((\R_+ \times (0, 1)) }+\| V^{m+1}\|_{L^{2q}((\R_+ \times (0, 1)) }\\
\leq c \norm{|V^mv^m|+|V^mv^{m-1}|+|V^mU|}_{L^{q}((\R_+ \times (0, 1))  }\\
\leq 3c A \norm{V^m}_{L^{2q}((\R_+ \times (0, 1)) }
<\frac12 \norm{V^m}_{L^{2q}((\R_+ \times (0, 1)) }.
\end{align*}
Therefore, $(v^m, \na v^m)$ converges to $(v, \na v)$ in $L^{2q}   (\R_+ \times (0, 1)) \times L^q   (\R_+ \times (0, 1))$ such that $v$
solves in the sense of distributions
\[
v_t-\Delta v+\nabla q=-{\rm div}\,\left(v\otimes v+v\otimes
U+v\otimes U+U\otimes U\right),
\]
\[
{\rm div} \, v =0
\]
with homogeneous initial and boundary data, i.e. $v(x,0)=0$ and
$v(x,t)=0$ on $\{x_n=0\}$. In addition, via \eqref{0531-1}, we have
\begin{align*}
\|\na  v\|_{L^{q}((\R_+ \times (0, 1)) }, \quad
 \|v \|_{L^{2q}((\R_+ \times (0, 1)) } \leq   A.
\end{align*}
Furthermore, if $ q > n+2$ and $\|w\|^2_{L^{2q}  } \leq  c \al$, it follows from  Proposition \ref{theoexternel-boundary} that
\begin{align}\label{0531}
\begin{split}
 \| v\|_{L^{\infty}(\R_+ \times (0, 2)) } &  \leq c
\||v|^2+|vU|+|U|^2\|_{L^{q}(\R_+ \times (0, 1)) }  \\
&\leq  2c\left(\|v\|^2_{L^{2q}((\R_+ \times (0, 1)) }+\|w\|^2_{L^{2q}  }\right)   \leq 4c A^2 <A.
\end{split}
\end{align}
 This complete the proof of Proposition \ref{theo1001}.
\qed

\section*{Acknowledgement}
T. Chang is supported by RS-2023-00244630. K. Kang is supported by RS-2024-00336346.



\end{document}